\documentclass[a4paper,report, 11pt]{article}
\usepackage{amsmath,a4wide}
\usepackage{amssymb}
\usepackage{amscd}
\usepackage{epsfig}
\usepackage{graphics}
\usepackage{multirow}

\usepackage{graphicx}

\usepackage{subfigure} 
\usepackage{epstopdf}
\usepackage{booktabs}  

\usepackage{ifpdf}
\usepackage[a4paper,top=1.5in, bottom=1.5in, left=0.5in, right=0.5in]{geometry}
\usepackage{changes}
\usepackage[makeroom]{cancel}
\usepackage{bbm}

\usepackage{amsthm}
\usepackage{color}
\definecolor{amaranth}{rgb}{0.9, 0.17, 0.31}	
\definecolor{myblue}{rgb}{0.2,0,0.9}
\definecolor{auburn}{rgb}{0.43, 0.21, 0.1}
\definecolor{bittersweet}{rgb}{1.0, 0.44, 0.37}
\definecolor{blue-violet}{rgb}{0.54, 0.17, 0.89}
\usepackage[authoryear]{natbib}
\usepackage{hyperref}

\hypersetup{
	colorlinks,linkcolor=myblue,citecolor=myblue,filecolor=magenta,urlcolor=myblue}

\newtheorem{pro}{Proposition}[section]
\newtheorem{pr}{Problem}

\newtheorem{theorem}{Theorem}[section]
\newtheorem{cor}{Corollary}[section]
\newtheorem{lem}{Lemma}[section]

\newtheorem{as}{Assumption}

\def\half{\frac{1}{2}}

\def\d{\delta}

\def\e{\epsilon}
\def\g{\gamma}

\def\m{\mu}
\def\n{\nu}

\def\s{\sigma}
\def\t{\theta}

\def\be{\begin{eqnarray}}
\def\ee{\end{eqnarray}}

\begin{document}
	
	\title{\LARGE {\bf A Unified Approach to Retirement  and Consumption-Portfolio Choice\footnote{ Junkee Jeon is supported by the National Research Foundation of Korea (NRF) grant funded by the Korea government [Grant No. NRF-2020R1C1C1A01007313]. Hyeng Keun Koo is supported by NRF grant [Grat No. NRF-2020R1A2C1A01006134].
	} }}
	
	\author{
		Junkee Jeon \footnote{E-mail: {\tt junkeejeon@khu.ac.kr}\;Department of Applied Mathematics, Kyung Hee University, Korea.} 
		\and
		Hyeng Keun Koo\footnote{E-mail: {\tt hkoo@ajou.ac.kr}\;Department of Financial Engineering, Ajou University, Korea.}
	}
	
	\date{\today}
	
	\maketitle \pagestyle{plain} \pagenumbering{arabic}
	
	\abstract{In this study  we propose a unified model of optimal retirement,  consumption and portfolio choice  of an individual agent, which encompasses a large class of the models in the literature and provide a general methodology to solve the model. Different from the traditional approach, we consider the problems before and after retirement simultaneously and identify the difference in the dual value functions as the utility value of lifetime labor. The utility value has an option nature, namely, it is the maximized value of choosing the retirement time optimally and we discover it by solving a variational inequality. Then, we discover the dual value functions by using the utility value. We discover the value function and optimal policies by establishing a duality between the value function and the dual value function.
		}
	
\vspace{1.0cm}
	
		{\bf JEL Classification Codes}:  D14, D15, E21, G11
	\medskip
	
	{\em Keywords} : Consumption, Human Wealth, Option Value, Portfolio, Retirement, Utility Function  \\

\newpage

\section{Introduction}

In this study, we propose and investigate a general model of optimal retirement,  consumption and portfolio choice  of an individual agent. The optimal  retirement decision is an important determinant of labor supply and human capital. The option of voluntary retirement  makes  the beta of human capital negative, since one can work longer after poor realizations of investment outcomes and has an important effect on risk taking attitudes (\citet{DL2010}). Researchers have developed economic models of the retirement decision in conjunction with the consumption and portfolio choice. In particular, they have shown that the retirement option induces people to increase savings and to make aggressive investments before retirement and to increase risky positions in a stock market boom even after making large profits (\citet{CS2006}, \citet{FP2007},  \citet{CSS2008}, and \citet{YK}).
The models, thus, have provided a profound insight into the interaction between the retiement decision and the life cycle choice of consumption and investment.

The models in the literature, however, are highly specialized and often make contradicting predictions. For example, the models of \citet{FP2007} and \citet{DL2010} predict jumps of consumption at retirement, whereas those of \citet{CS2006} and \citet{YK} do not. The difference in the models originates from the special utility functions they employ, and hence, the predictions are not robust to changes in the utility functions.

Our objective is (i) to propose a unified model, which encompasses a large class of the models in the literature and (ii) to provide a general methodology to solve the model. For the first objective, we propose a model with  minimum assumptions among those employing time separable expected utility functions with a constant subjective discount rate and stationary felicity functions before and after retirement in a single-good economy. The minimum assumptions consist of two;  Assumption \ref{as:utility}   is the condition for the problem to be well-defined, which is much less restrictive than those in the literature,  and  Assumption \ref{as:limit-psi} is  the necessary and sufficient condition  for the retirement option to have a positive value. We next solve the problem by the dual martingale approach.  However, we  do not take the traditional approach of taking sequential steps; solving the problem after retirement in the first step and then solving the problem before retirement in the second step, taking the solution in the second step as an essential ingredient.\footnote{Most of the papers in the literature take this two-step approach. See \citet{YK} for a systematic treatment.} Instead we consider the problems before and after retirement simultaneously and identify the difference in the dual value functions as the utility value of lifetime labor. The utility value has an option nature, namely, it is the maximized value of choosing the retirement time optimally, and can be obtained as a solution to a variational inequality. Then, we find the dual value functions by using the utility value. We discover the value function and optimal policies by establishing a duality between the value function and the dual value function. In this way we show the total wealth of the agent is composed of financial wealth and human wealth and human wealth can be derived as a part of the derivative of utility value of lifetime labor.

As an illustrative example we provide a model which subsumes major models in the literature as special cases.  We obtain optimal policies in closed form when the felicity function has constant relative risk aversion.

\citet{YK} have considered the utility value of lifetime labor and derived its value by solving a variational inequality. Their method is, however, still the  two-step method.  \citet{JP2020} make an observation that  the dual value functions before and after retirement differ by the the utility value of lifetime labor. Their utility functions  before and after retirement differ by an additive disutility, and we consider generally different utility functions in this paper. In the absence of voluntary retirement option  
\citet{BMS1992} and \citet{BDOW2004} study flexible labor supply with an exogenously fixed retirement date.

The paper is organized as follows. Section \ref{sec:model} explains the model. 
Section \ref{sec:optimization} sets up the dual optimization problem and provides a solution. Section \ref{sec:optimal_policies} establishes the duality between the value function and the dual value function and derives the optimal policies. Section \ref{sec:examples} discusses examples and Section \ref{sec:conclusion} concludes.

\section{Model}\label{sec:model}

We consider an agent who lives in a single good economy. The agent is currently working and has an option to retire voluntarily. The agent cannot come  back to work after retirement, and hence,  retirement is an irreversible decision. The agent's preference over consumption is represented by the following expected utility function:
	\begin{equation}\label{eq:utility}
	U\equiv \mathbb{E}\left[\int_0^\tau e^{-\rho t} u_B(c_t)dt+\int_\tau^\infty e^{-\rho t}u_A(c_t)dt\right].
\end{equation}
Here $\rho >0$ is the subjective discount rate, $\tau$ is the time of retirement, $c_t$ is the rate of consumption, and $u_B: {\mathcal D}\rightarrow \mathbb{R}$ and $u_A: {\mathcal D}\rightarrow \mathbb{R}$ (${\mathcal D}= [0,\infty)$ or $(0,\infty)$) are the agent's felicity function of before and after retirement, respectively.\footnote{  We consider an infinite horizon model for simplicity of exposition, and extension to a finite horizon is technically complex, but can be done following the methods in \citet{YK}.
}
We define the class $\Sigma$ of all set of felicity function $u$ satisfying the following conditions:
\begin{itemize}
	\item[(i)] The felicity function $u:{\mathcal D}\to \mathbb{R}$ is strictly increasing, strictly concave and continuously differentiable, and $\lim_{c\to +\infty}u'(c) = 0$. \vspace{2mm}
	
	The strictly decreasing and continuous function $u':(0,\infty) \overset{\textrm{onto}}{\longrightarrow} (0,u'(0))$ has a strictly decreasing, continuous inverse $I_u:(0,u'(0))\overset{\textrm{onto}}{\longrightarrow}(0,\infty)$. We extend $I_u$ by setting $I_u(y)=0$ for $y\ge u'(0)$. Then, we have 
	\begin{align}\label{eq:inverse-extend}
	u'(I_u(y))=\begin{cases}
	y,\;\;\;&0<y<u'(0),\\
	u'(0),\;\;\;&y\ge u'(0),
	\end{cases}
	\end{align}
	and $I_u(u'(c))=c$ for $0<c<\infty$. Note that $\lim_{y\to \infty}I_u(y) = 0$.
	\item[(ii)] 	For any $y>0$, 
	\begin{align}\label{as:well-defined}
	\int_0^y \eta^{-n_2}I_u(\eta)d\eta < \infty.
	\end{align}
	where	$n_1>0$ and $n_2<0$ are two roots of the quadratic equation:
	\begin{equation}\label{eq:quadratic}
	\dfrac{\t^2}{2}n^2 + \left(\rho-r-\frac{\t^2}{2}\right)n-\rho=0.\;\;\;
	\end{equation}
\end{itemize}

To guarantee that the problem is well-defined, we make the following assumption:
\begin{as}\label{as:utility}
	\item[(i)] $u_B,u_A \in \Sigma$, 
	\item[(ii)] for all $y>0$, the following inequality holds:
	\begin{equation*}
	u_B(I_{u_B}(y))<u_A(I_{u_A}(y)).
	\end{equation*}

\end{as}

There are two assets in the financial market, a risk-free asset and a risky asset. The return on the risk-free asset is a constant $r>0$. The cum-dividend price $S_t$ of the risky asset follows the dynamics:
\be\label{eq:dynamics_risky}
\dfrac{dS_t}{S_t} = \mu dt+ \s dB_t,
\ee
where $\mu>0$ and $\sigma$ are constants, describing the mean and the standard deviation of the returns on the risky asset, and $B_t$ is a standard Brownian motion defined on a filtered probability space $(\Omega,{\mathcal F}, P)$.  We assume that  filtration ${\mathcal F}=({\mathcal F}_t)_{t\geq 0}$ is the augmented filtration generated by $B_t$; it describes the information available to the agent at each instance. The two asset model is a simplification of that of  $n$-risky assets and a single risk-free asset with constant covariance matrix of risky asset returns, where the capital asset pricing model(CAPM) and the two-fund separation theorem are valid (\citet{GL}.

We assume that the agent receives constant wage income equal to $\e>0$ until retirement. The constant wage rate assumption is standard in literature.\footnote{For example, \citet{FP2007} states, ``we maintain the assumption that agents receive a constant wage.
	This is done not only for simplicity, but more importantly because it makes the results
	more surprising." (p.91)} Let $\pi_t$ denote the agent's investment in the risky asset at time $t$. Then, 
her wealth $X_t$ at time $t$ satisfies the dynamics:
	\begin{equation}\label{eq:wealth_dynamics}
	dX_t = [r X_t + (\m-r)\pi_t + \e {\bf 1}_{\{t<\tau\}} - c_t]dt + \s \pi_t dB_t \;\;\mbox{with}\;\;X_0=x,
\end{equation}
where $x$ is her initial wealth {and ${\bf 1}_A$ denotes the indicator function of set $A$.  The natural limit for the wealth  is  the following:
\be\label{eq:credit_constraint}
X_t \geq -\frac{\e}{r} {\bf 1}_{\{t< \tau\}},
\ee
that is the agent cannot borrow more than the maximum possible present value of labor income.
}

 {We give technical admissibility conditions for {$(c,\pi,\tau)$}.  Throughout the paper, $(c,\pi,\tau)$ belongs to the admissible class ${\cal A}(x)$ if they are ${\cal F}_t$-progressively measurable processes satisfying the following conditions:
\begin{itemize}
	\item[(a)] $\tau$ belongs to ${\cal S}$,  the set of all ${\cal F}$-stopping times taking values in $(0,\infty)$,
	\item[(b)] $(c,\pi,\tau)$ satisfies the agent's wealth dynamics in \eqref{eq:wealth_dynamics} together with  constraint \eqref{eq:credit_constraint},
	\item[(c)] $c_t$ and  $\pi_t$ satisfy 
	\begin{equation}\label{con:feasible}
	\int_0^t c_s ds <\infty, \;\;\mbox{a.s}\;\;\mbox{and}\;\;\int_0^t \pi_s^2 ds <\infty\;\;\mbox{a.s},\;\;\;\forall\;t\ge 0.
\end{equation}
\end{itemize}
}

We now state the agent's problem:
\begin{pr}\label{pr:optimization_problem} Given $x>-\frac{\e}{r}$, we consider the following optimization problem:
\begin{equation*}
V(x) = \sup_{(c,\pi,\tau)\in{\cal A}(x)} \mathbb{E}\left[\int_0^\tau e^{-\rho t} u_B(t) dt+\int_\tau^\infty e^{-\rho t}u_A(c_t)dt\right]. 
\end{equation*}
\end{pr}

\section{Optimization}\label{sec:optimization}

We derive a solution to the agent's optimization problem by the standard dual martingale approach developed by \citet{KLS} and \citet{CoxH}. They show that the stochastic discount factor of the economy takes the form:
\be\label{stochastic_discount_factor}
{\xi}_t = e^{-rt-\half \theta ^2 t -\theta B_t}, \quad \theta \equiv \frac{\mu -r}{\s}.
\ee
Note that the constant $\theta$ is equal to the Sharpe ratio of the risky asset.  



In  
the following proposition, we  transform the wealth dynamics \eqref{eq:wealth_dynamics} into a budget constraint in static form. 

\begin{pro}\label{pro:budget-equaltiy}
	Let $x>-\frac{\e}{r}.$ Suppose $(c,\pi,\tau)\in {\cal A}(x).$ Then  $c=(c_t)_{t=0}^\infty$  satisfies the following static budget constraint 
	\begin{equation}\label{static_budget_constraint}
	\mathbb{E}\left[\int_0^\infty \xi\left(c_t-\e {\bf 1}_{\{t<\tau\}}\right)dt \right] =x.
	\end{equation}
Conversely, suppose that  $c=(c_t)_{t=0}^\infty$ satisfies the first condition in \eqref{con:feasible} and $\tau\in{\cal S}$ such that \eqref{static_budget_constraint} is satisfied. 	Then there exists a portfolio process $\pi_t$ such that $(c,\pi,\tau)\in {\cal A}(x)$. The corresponding wealth process $X^{x,c,\pi}$ is 
	\begin{equation*}
	dX_t^{x,c,\pi} = [r X_t^{x,c,\pi} +(\m-r)\pi_t -c_t +\e{\bf 1}_{\{t<\tau\}}]dt +\s \pi_t dB_t \;\;\;t\ge 0
	\end{equation*}
	and 
	\begin{equation}
	X_t^{x,c,\pi} =\mathbb{E}_t\left[\int_t^\infty \dfrac{ \xi_s}{ \xi_t} (c_s-\e 1_{\{s<\tau \}})ds\right],
	\end{equation}
	where $\mathbb{E}_t[\cdot]=\mathbb{E}\left[\cdot \mid {\cal F}_t\right]$ is the conditional expectation at time $t$ on the filtration ${\cal F}_t$. 
\end{pro}

\subsection{Properties of the class $\Sigma$}

For $y>0$, the {conjugate function} $\tilde{u}$ of $u\in \Sigma$ is defined as 
\begin{equation*}
\tilde{u}(y) = \sup_{c\ge 0}\left(u(c)-yc\right) = u(I_u(y))- yI_u(y).
\end{equation*}

\begin{lem}\label{lem:finite}
	For the felicity function $u\in \Sigma$ , the following integrability conditions hold:
	\begin{equation*}
	\int_0^y \nu^{-n_2-1} |u(I_u(\nu))|d\nu +\int_y^\infty \nu^{-n_1-1}|u(I_u(\nu))|d\nu <\infty. 
	\end{equation*}
	and
	\begin{equation*}
	\int_0^y \nu^{-n_2-1} |\tilde{u}(\nu)|d\nu +\int_y^\infty \nu^{-n_1-1}|\tilde{u}(\nu)|d\nu <\infty. 
	\end{equation*}
\end{lem}

\begin{proof}
	
	It is easy to check that for any $u\in \Sigma$ 
	\begin{equation*}
	\int_0^y \nu^{-n_2}I_{u}(\nu)d\nu +\int_y^\infty \nu^{-n_1}I_u(\nu)d\nu <\infty. 
	\end{equation*}

	Note that for any $y>0$ 
\begin{equation*}
u(I_u(y))-yI_u(y)=\tilde{u}(y)\;\;\mbox{and}\;\;\tilde{u}'(y)=-I_u(y),
\end{equation*}
and thus
\begin{equation*}
u(I_u(\nu))=u(I_u(y))-yI_u(y)+\nu I_u(\nu)+\int_{\nu}^y I_u(\eta)d\eta. 
\end{equation*}

	It follows that 
	\begin{footnotesize}
		\begin{eqnarray}
		\begin{split}\label{eq:u-integral}
		&\int_0^y \nu^{-n_2-1}|u(I_{u}(\nu))|d\nu +\int_y^\infty \nu^{-n_1-1}|u(I_{u}(\nu))|d\nu\\
		\le-&\dfrac{1}{n_2}y^{-n_2}|u(I_{u}(y))-yI_{u}(y)|+\int_0^y \nu^{-n_2}I_{u}(\nu)d\nu +\int_0^y \int_\nu^y \nu^{-n_2-1}I_{u}(\eta)d\eta d\nu\\ +&\dfrac{1}{n_1}y^{-n_1}|u(I_{u}(y))-yI_{u}(y)|+\int_y^\infty \nu^{-n_1}I_{u}(\nu)d\nu +\int_y^\infty \int_y^\nu \nu^{-n_1-1}I_{u}(\eta)d\eta d\nu\\
		=-&\dfrac{1}{n_2}y^{-n_2}|u(I_{u}(y))-yI_{u}(y)|+\dfrac{1}{n_1}y^{-n_1}|u(I_{u}(y))-yI_{u}(y)|\\&+\left(1+\dfrac{1}{n_1}\right)\int_y^\infty \nu^{-n_1}I_{u}(\nu)d\nu +\left(1-\dfrac{1}{n_2}\right)\int_0^y \nu^{-n_2}I_{u}(\nu)d\nu <\infty. 
		\end{split}
		\end{eqnarray}
	\end{footnotesize}
	where we have used Fubini's theorem in last equality.
	
	Since $u(I_u(\nu))=\tilde{u}(\nu)+\nu I_{u}(\nu)$, we have 
	\begin{equation}\label{eq:utilde-integral}
	\int_0^y \nu^{-n_2-1}|\tilde{u}(\nu)|d\nu +\int_y^\infty \nu^{-n_1-1}|\tilde{u}(\nu)|d\nu<\infty.
	\end{equation}
\end{proof}

For any measurable function $f:\mathbb{R}^+\to \mathbb{R}$, we define two operators $\Xi$ and $\Gamma$ by 
\begin{equation*}
\Xi_f(y)=\dfrac{2}{\t^2(n_1-n_2)}\left[y^{n_2}\int_0^y \nu^{-n_2-1}f(\nu)d\nu +y^{n_1}\int_y^\infty \nu^{-n_1-1}f(\nu)d\nu\right]
\end{equation*}
and
\begin{equation*}
\Gamma_f(y)=\dfrac{2}{\t^2(n_1-n_2)}\left[y^{n_2-1}\int_0^y \nu^{-n_2}f(\nu)d\nu +y^{n_1-1}\int_y^\infty \nu^{-n_1}f(\nu)d\nu\right].
\end{equation*}

For $u\in \Sigma$, $\Gamma_{I_u},\;\Xi_{u\circ I_u}$ and $\Xi_{\tilde{u}}$ are given by
\begin{align*}
\Gamma_{I_u}(y)=&\dfrac{2}{\t^2(n_1-n_2)}\left[y^{n_2-1}\int_0^y \nu^{-n_2}I_u(\nu)d\nu +y^{n_1-1}\int_y^\infty \nu^{-n_1}I_u(\nu)d\nu\right],\\
\Xi_{u\circ I_u}(y)=&\dfrac{2}{\t^2(n_1-n_2)}\left[y^{n_2}\int_0^y \nu^{-n_2-1}u(I_u(\nu))d\nu +y^{n_1}\int_y^\infty \nu^{-n_1-1}u(I_u(\nu))d\nu\right]
\end{align*}
and 
\begin{equation*}
\Xi_{\tilde{u}} (y)=\dfrac{2}{\t^2(n_1-n_2)}\left[y^{n_2}\int_0^y \nu^{-n_2-1}\tilde{u}(\nu)d\nu +y^{n_1}\int_y^\infty \nu^{-n_1-1}\tilde{u}(\nu)d\nu\right],
\end{equation*}
respectively. 

By Lemma \ref{lem:finite} and Proposition \ref{pro:review:KMZ}, we can directly obtain the following proposition:
\begin{pro} For any $u\in \Sigma$, $\Gamma_{I_u},\;\Xi_{u\circ I_u}$ and $\Xi_{\tilde{u}}$ are well-defined and the following equalities hold:
	\begin{equation*}
	\Gamma_{I_u}(y)=\dfrac{1}{y}\mathbb{E}\left[\int_0^\infty e^{-\rho t}{\cal Y}_t^y I_u({\cal Y}_t^y)dt\right],\;\;	\Xi_{u\circ I_u}(y)=\mathbb{E}\left[\int_0^\infty e^{-\rho t}u(I_u({\cal Y}_t^y))dt\right]
	\end{equation*}
	and
	\begin{equation*}
	\Xi_{\tilde{u}}(y)=\mathbb{E}\left[\int_0^\infty e^{-\rho t}\tilde{u}({\cal Y}_t^y)dt\right],
	\end{equation*}
	where ${\cal Y}_t^y \equiv y e^{\rho t}\xi_t$. 
\end{pro}

\begin{lem}\label{lem:Xi-Gamma}
	For any $u\in \Xi$, the following statements are true:
	\begin{itemize}
		\item[(a)] $\Xi_{\tilde{u}}'(y) = -\Gamma_{I_u}(y)$. 
		\item[(b)] $\Gamma_{I_u}(y)$ is strictly decreasing in $y>0$. Thus, $\Xi_{\tilde{u}}(y)$ is strictly convex in $y>0$. 
		\item[(c)] $\lim_{y\to \infty} \Gamma_{I_u}(y)=0$ and $\lim_{y\to 0+}\Gamma_{I_u}(y)=\infty$. 
		
		\item[(d)] $\lim_{t\to\infty} e^{-\rho t} \mathbb{E}\left[{\cal Y}_t^y \Gamma_{I_u}({\cal Y}_t^{y^*})\right]=0$. 
		
		\item[(e)] There exist positive constants $\zeta_1$ and $\zeta_2$ such that 
		\begin{equation*}
		|y\Gamma_{I_u}'(y)| \le \zeta_1 (y^{\zeta_2}+y^{-\zeta_2}). 
		\end{equation*}
	\end{itemize}
\end{lem}
\begin{proof}
	\noindent{\bf (a):} By the definition of $\Xi_{\tilde{u}}(y)$, we have 
	\begin{align*}
	\Xi_{\tilde{u}}'(y)=\dfrac{2}{\t^2(n_1-n_2)}\left[n_2 y^{n_2-1}\int_0^y \nu^{-n_2-1}\tilde{u}(\nu)d\nu +n_1y^{n_1-1}\int_y^\infty \nu^{-n_1-1}\tilde{u}(\nu)d\nu\right]. 
	\end{align*}
	It follows from Lemma \ref{lem:finite} and Proposition \ref{pro:review:KMZ} that 
	\begin{align}\label{eq:limit-u-tilde}
	\liminf_{y\downarrow 0}y^{-n_2}\tilde{u}(y) = \liminf_{y\uparrow \infty} y^{-n_1}\tilde{u}(y)=0. 
	\end{align}
	By applying the integration by parts for Riemann-Stieltjes integral, it follows from the limiting behavior of $\tilde{u}$ in \eqref{eq:limit-u-tilde} that 
	\begin{align*}
	\Xi_{\tilde{u}}'(y)=&\dfrac{2}{\t^2(n_1-n_2)}\left[n_2 y^{n_2-1}\int_0^y \nu^{-n_2-1}\tilde{u}(\nu)d\nu +n_1y^{n_1-1}\int_y^\infty \nu^{-n_1-1}\tilde{u}(\nu)d\nu\right]\\
	=&-\dfrac{2}{\t^2(n_1-n_2)}\left[y^{n_2-1}\int_0^y \nu^{-n_2}I_u(\nu)d\nu +y^{n_1-1}\int_y^\infty \nu^{-n_1}I_u(\nu)d\nu\right]\\
	=&-\Gamma_{I_u}(y),
	\end{align*}
	where we have used the fact $(\tilde{u}(y))'=-I_u(y)$.\\
	
	\noindent{\bf (b):} It follows from $u\in \Sigma$ and Proposition \ref{pro:review:KMZ} that 
	\begin{align}\label{eq:limit-I-u}
	\liminf_{y\downarrow 0}y^{-n_2}I_u(y) = \liminf_{y\uparrow \infty} y^{-n_1}I_u(y)=0. 
	\end{align}
	
	The integration by parts for Riemann-Stieltjes integral implies that 
	\begin{align*}
	\Gamma_{I_u}'(y)=&\dfrac{2}{\t^2(n_1-n_2)}\left[(n_2-1)y^{n_2-2}\int_0^y \nu^{-n_2}I_u(\nu)d\nu +(n_1-1)y^{n_1-2}\int_y^\infty \nu^{-n_1}I_u(\nu)d\nu\right]\\
	=&\dfrac{2}{\t^2(n_1-n_2)}\left[y^{n_2-2}\int_0^y \nu^{1-n_2}I^{\prime}_u(\nu)d\nu +y^{n_1-2}\int_y^\infty \nu^{1-n_1}I^{\prime}_u(\nu)d\nu\right]<0,
	\end{align*}
	where we have used the fact $I_u$ is strictly decreasing in $y>0$.\\
	
	\noindent{\bf (c):} Note that 
	\begin{equation*}
	\Gamma_{I_u}(y)=\dfrac{1}{y}\mathbb{E}\left[\int_0^\infty e^{-\rho t}{\cal Y}_t^y I_u({\cal Y}_t^y)dt\right]=\mathbb{E}\left[\int_0^\infty \xi_t I_y({\cal Y}_t)dt\right]. 
	\end{equation*}
	
	Since $\mathbb{E}\left[\int_0^\infty \xi_t I_u({\cal Y}_t^y)dt \right]<\infty$, the dominated convergence theorem implies that 
	\begin{align*}   
	\lim_{y\downarrow 0}\Gamma_{I_u}(y) =\mathbb{E}\left[\int_0^\infty \lim_{y\downarrow 0} \xi_t I_u(ye^{\rho t}{\xi}_t)dt\right]=\infty
	\end{align*}
	and 
	\begin{align*}
	\lim_{y\uparrow \infty}\Gamma_{I_u}(y) =\mathbb{E}\left[\int_0^\infty \lim_{y\uparrow \infty} \xi_t I_u(ye^{\rho t}{\xi}_t)dt\right]=0.
	\end{align*}
	
	\noindent {\bf Proof of (d):} Note that 
	\begin{equation*}
	y\Gamma_{I_u}(y) = \mathbb{E}\left[\int_0^\infty e^{-\rho t}{\cal Y}_t^y I_u({\cal Y}_t^y)dt\right]<\infty.
	\end{equation*}
	
	Thus, it follows from Proposition \ref{pro:review:KMZ} (e) that 
	$$\lim_{t\to\infty} e^{-\rho t} \mathbb{E}\left[{\cal Y}_t^y \Gamma_{I_u}({\cal Y}_t^{y^*})\right]=0.$$\\
	
	\noindent {\bf Proof of (e):} 
	Since $\Gamma_{I_u}(y) = -\Xi_{\tilde{u}}'(y)$, Proposition \ref{pro:review:KMZ} (d) implies that there exits a constant $C_1>0$ satisfying  
	\begin{equation}\label{eq:Gamma-prime-1}
	|\Gamma_{I_u}(y)| \le C_1 (y^{n_1-1}+y^{n_2-1}). 
	\end{equation}
	
Since $y\Gamma_{I_u}(y) = \mathbb{E}\left[\int_0^\infty e^{-\rho t}{\cal Y}_t^y I_u({\cal Y}_t^y)dt\right]$,  there exists a constant $C_2>0$ such that 
	\begin{equation}\label{eq:Gamma-prime-2}
	|(y\Gamma_{I_u}(y))'| \le C_2 (y^{n_1-1}+y^{n_2-1}).
	\end{equation}
	 
	Thus, we have 
	\begin{align*}
	|y\Gamma_{I_u}'(y)|\le |(y\Gamma_{I_u}(y))'| + |\Gamma_{I_u}(y)| \le (C_1+C_2)(y^{n_1-1}+y^{n_2-1}).
	\end{align*}
	
	This completes the proof. 
	
\end{proof}
%
%
%
%
%
%

\subsection{Dual formulation}

We  formulate the following Lagrangian: 
\begin{eqnarray}
\begin{split}\label{eq:Lagrangian}
\mathfrak{L}=&\mathbb{E}\left[\int_0^\tau e^{-\rho t} u_B(c_t)dt+\int_\tau^\infty e^{-\rho t}u_A(c_t)dt\right]+y\left(x-\mathbb{E}\left[\int_0^\infty {\xi}_t (c_t -\e{\bf 1}_{\{t<\tau\}})dt \right]\right)\\
=&\mathbb{E}\left[\int_0^\tau e^{-\rho t}(u_B(c_t)-{\cal Y}_t^y c_t +{\cal Y}_t^y \e)dt + e^{-\rho \tau}{\mathbb{E}_{\tau}\left[\int_{\tau}^\infty e^{-\rho(t-\tau)}\left(u_A(c_t)-{\cal Y}_t^y c_t\right)dt\right]}\right] +yx \\
\le&\mathbb{E}\left[\int_0^\tau e^{-\rho t}(u_B(c_t)-{\cal Y}_t^y c_t +{\cal Y}_t^y \e)dt + e^{-\rho \tau}{\mathbb{E}_{\tau}\left[\int_{\tau}^\infty e^{-\rho(t-\tau)}\tilde{u}_A({\cal Y}_t^y)dt\right]}\right] +yx \\
\le&\mathbb{E}\left[\int_0^\tau e^{-\rho t}\left(\tilde{u}_B({\cal Y}_t^y )+{\cal Y}_t^y \e\right)dt +\int_\tau^\infty e^{-\rho t}\tilde{u}_A({\cal Y}_t^y)dt\right] +yx\\
=&\mathbb{E}\left[\int_0^\tau e^{-\rho t}\left(\tilde{u}_B({\cal Y}_t^y )-\tilde{u}_A({\cal Y}_t^y)+{\cal Y}_t^y \e\right)dt +\int_0^\infty e^{-\rho t}\tilde{u}_A({\cal Y}_t^y)dt\right] +yx\\
=&J_A(y) + \mathbb{E}\left[\int_0^\tau e^{-\rho t}\left(\tilde{u}_B({\cal Y}_t^y )-\tilde{u}_A({\cal Y}_t^y)+{\cal Y}_t^y \e\right)dt\right] +yx,
\end{split}
\end{eqnarray}
where $y>0$ is a Lagrangian multiplier of the static budget constraint \eqref{static_budget_constraint}, ${\cal Y}_t^y = y e^{\rho t} {\xi}_t$,
\be\label{de:dual_vf_a}
 J_A(y)=\mathbb{E}\left[\int_0^\infty e^{-\rho t}\tilde{u}_A({\cal Y}_t^y)dt\right],
\ee
and $\tilde{u}_B$, $\tilde{u}_A$ are the conjugate function of $u_B$, $u_A$ given by 
	\begin{align*}
	\tilde{u}_B(y)\equiv&\sup_{c\ge 0}\left(u_B(c)-yc\right)=u_B(I_{u_B}(y))-yI_{u_B}(y),\\
	\tilde{u}_A(y)\equiv&\sup_{c\ge 0}\left(u_A(c)-yc\right)=u_A(I_{u_A}(y))-yI_{u_B}(y),
	\end{align*}
respectively. The dual variable ${\cal Y}_t$ is the {\em marginal value} of wealth, as will be shown in Theorem \ref{thm:main}. 

Thus, for $y>0$ the candidate of optimal consumption $(\hat{c}({\cal Y}_t^y))_{t=0}^{\infty}$ is given by 
{\begin{eqnarray}
\begin{split}\label{eq:candi-opti}
\hat{c}({\cal Y}_t^y)  =  
\begin{cases}
I_{u_B}({\cal Y}_t^y)\;\;\;&\mbox{for}\;\;0\le t <\tau,\\
I_{u_A}({\cal Y}_t^y)\;\;\;&\mbox{for}\;\;t\ge \tau.
\end{cases}
\end{split}
\end{eqnarray}}

Considering the Lagrangian \eqref{eq:Lagrangian}, we formulate the dual problem which chooses the optimal time of retirement by observing the changes of the marginal value of wealth ${\cal Y}_t$.
\begin{pr}\label{pr:optimal_stopping}
	\begin{equation}\label{eq:dual_value}
	J(y) = J_A(y) + \sup_{\tau \in {\cal S}} \mathbb{E}\left[\int_0^\tau e^{-\rho t}\left(\tilde{u}_B({\cal Y}_t^y )-\tilde{u}_A({\cal Y}_t^y)+{\cal Y}_t^y \e\right)dt\right],
	\end{equation}
where ${\cal S}$ denotes the set of ${\mathcal F}-$stopping times. We call  $J(y)$ the {\em dual value function}. 
\end{pr}

Problem \ref{pr:optimal_stopping} states that the dual value function consists of $J_A$ and the optimized value of the optimal stopping problem. If $\tau=0$, i.e., the agent retires immediately, then the dual value function is equal to $J_A$, and hence, $J_A$ can be regarded as the {\em dual value function after retirement}. The quantity
\be\label{de:utlity_working}
\tilde{u}_B({\cal Y}_t^y )-\tilde{u}_A({\cal Y}_t^y)+{\cal Y}_t^y \e
\ee
 inside the integral in the optimal stopping problem is composed of two components: the first is the difference in the conjugate felicity functions, 
$\tilde{u}_B({\cal Y}_t^y )-\tilde{u}_A({\cal Y}_t^y),$ which can be interpreted as the difference in the utility values before and after retirement, and the second is ${\cal Y}_t^y \e$, labor income adjusted by the marginal utility of wealth and can be interpreted as the utility value of labor income. Accordingly,  quantity \eqref{de:utlity_working} is the {\em marginal benefit of work} relative to retirement. The agent  chooses the retirement time $\tau$ that maximizes the present value of the marginal benefit. In this sense we will call the optimized value of the optimal stopping problem as the {\em utility value of lifetime labor.}

Let us denote the  utility value of lifetime labor by ${\cal P}(y),$ i.e., 
\begin{equation}\label{eq:OSP}
{\cal P}(y) = \sup_{\tau \in {\cal S}} \mathbb{E}\left[\int_0^\tau e^{-\rho t}\left(\tilde{u}_B({\cal Y}_t^y )-\tilde{u}_A({\cal Y}_t^y)+{\cal Y}_t^y \e\right)dt\right].
\end{equation}
Let us denote the current normalized marginal benefit of work  by $\Psi(y),$ i.e., 
\begin{equation*}
\Psi(y) = \dfrac{1}{y}\left(\tilde{u}_B(y)-\tilde{u}_A(y)\right)+\e. 
\end{equation*}

We will now investigate properties of the normalized marginal benefit.
\begin{lem}\label{lem:psi}
	$\Psi(y)$ is a strictly increasing function of $y>0$ and 
	\begin{equation*}
	\lim_{y\to \infty} \Psi(y) = \e. 
	\end{equation*}
\end{lem}
\begin{proof}
	It follows from Assumption \ref{as:utility} that 
	\begin{align*}
	\Psi'(y)=& -\dfrac{1}{y^2}\left(\tilde{u}_B(y)-\tilde{u}_A(y)\right)+\dfrac{1}{y}(-I_{u_B}(y)+I_{u_A}(y))\nonumber\\
	=&-\dfrac{1}{y^2}\left((\tilde{u}_B(y)+yI_{u_B}(y))-(\tilde{u}_A(y)+yI_{u_A}(y))\right)\\
	=&-\left(u(I_{u_B}(y))-u(I_{u_A}(y))\right)>0.\nonumber
	\end{align*}
	Hence, 	$\Psi(y)$ is a strictly increasing function of $y>0$.
	
	\begin{itemize}
		\item[(i)] If $\lim_{y\to \infty} |\left(\tilde{u}_B(y)-\tilde{u}_A(y)\right)|=\infty$,
		
		it follows form L'Hospital's rule that
		\begin{align*}
		\lim_{y\to \infty}\dfrac{\tilde{u}_B(y)-\tilde{u}_A(y)}{y}=\lim_{y\to\infty}\dfrac{-I_{u_B}(y)+I_{u_A}(y)}{1}=0. 
		\end{align*}
		
		\item[(ii)] If $\lim_{y\to \infty} |\left(\tilde{u}_B(y)-\tilde{u}_A(y)\right)|<\infty$, 
		
		it is clear that 
		\begin{align*}
		\lim_{y\to \infty}\dfrac{\tilde{u}_B(y)-\tilde{u}_A(y)}{y}=0. 
		\end{align*}
	\end{itemize}

By (i) and (ii), we deduce that 
$$
\lim_{y\to \infty}\Psi(y) = \e. 
$$
\end{proof}

If $\lim_{y\to 0+} \Psi(y) \ge 0$, Lemma \ref{lem:psi} implies that $\Psi(y)>0$ for all $y>0$. Thus, we deduce that for any $\tau\in{\cal S}$
\begin{equation*}
\mathbb{E}\left[\int_0^\tau e^{-\rho t}\left(\tilde{u}_B({\cal Y}_t^y )-\tilde{u}_A({\cal Y}_t^y)+{\cal Y}_t^y \e\right)dt\right]\le \mathbb{E}\left[\int_0^\infty e^{-\rho t}\left(\tilde{u}_B({\cal Y}_t^y )-\tilde{u}_A({\cal Y}_t^y)+{\cal Y}_t^y \e\right)dt\right].
\end{equation*}

That is, the agent does not choose the option to retire in the case when $\lim_{y\to 0+} \Psi(y) \ge 0$, that is, the marginal benefit of work is always positive. Consequently, the necessary condition for the agent to retire at a finite time $\tau$ is  the following assumption:
\begin{as}\label{as:limit-psi}
	$$\lim_{y\to0+}\Psi(y) <0.$$
\end{as}

Under Assumption \ref{as:limit-psi}, Lemma \ref{lem:psi} implies that there exists a unique $\bar{z}>0$ such that 
$$
\Psi(\bar{z})=0. 
$$

By the standard theory of the optimal stopping problem (\citet{PS}), ${\cal P}(y)$ satisfies the following variational inequality:
\begin{align}
\begin{cases}\label{eq:VI}
\mathbb{L}{\cal P}+h(y) \le 0\;\;\;&\mbox{if}\;\;{\cal P}(y)=0, \vspace{2mm}\\
\mathbb{L}{\cal P}+h(y) = 0\;\;\;&\mbox{if}\;\;{\cal P}(y)>0,\vspace{2mm}\\
\end{cases}
\end{align}
 where $h(y)\equiv y\Psi(y)=(\tilde{u}_B(y)-\tilde{u}_A(y)+\e y)$ and the differential operator $\mathbb{L}$ is given by 
\begin{equation*}
\mathbb{L}\equiv \dfrac{\t^2}{2}y^2\dfrac{d^2}{dy^2} + (\rho-r)y\dfrac{d}{dy} - \rho.
\end{equation*}

\begin{theorem}[Verification theorem]~\label{thm:verification-OSP}\\
Suppose that  variational inequality \eqref{eq:VI} has a solution ${\cal Q}(y)$ which is continuously differentiable in $y>0$ and twice continuously differentiable in $(0,\infty)\backslash\{b\}$ for some point $b>0$, and there exist positive constants $\zeta_1$ and $\zeta_2$ such that 
\begin{equation*}
|{\cal Q}'(y)|\le \zeta_1\left(y^{-\zeta 2}+y^{\zeta_2}\right)\;\;\mbox{for}\;\;y>0. 
\end{equation*}
	Then, 
	\begin{itemize}
		\item[(a)] ${\cal P}(y) \le {\cal Q}(y)$.
		\item[(b)] If $\lim_{t\to \infty} e^{-\rho t} \mathbb{E}\left[{\cal Q}({\cal Y}_t)\right]=0$, then ${\cal Q}(y)={\cal P}(y)$ and the optimal solution to the problem in \eqref{eq:OSP} is given by 
		$$
		\tau_R(y) = \inf \{t\ge0 \mid {\cal Y}_t^y\in \{y>0\mid {\cal P}(y)=0 \} \}. 
		$$
	\end{itemize}
\end{theorem}
\begin{proof}
	 Even though ${\cal Q}(y)$ is only continuously differentiable, we can still apply It\^{o}'s lemma (see Exercise 6.24 in \citet{KS2}) to the process $\{e^{-\rho t}{\cal Q}({\cal Y}_t^y)\}$. The proof of this theorem is almost similar to that of Theorem 3.2 and Lemma 3.4 in \citet{KMZ} and  we omit the proof.
\end{proof}

From Theorem \ref{thm:verification-OSP}, we know that the transversality condition is necessary to guarantee that the solution to the HJB equation is the option value of voluntary retirement:
\begin{equation}\label{eq:transversality}
\lim_{t\to \infty} e^{-\rho t} \mathbb{E}\left[{\cal P}({\cal Y}_t^y)\right]=0.
\end{equation}

We provide the utility value of labor in the following proposition: 
\begin{pro}\label{pro:OS}
The utility value of lifetime labor,	${\cal P}(y)$,  is given by 
		\begin{eqnarray}\label{eq:P-closed}
		\begin{split}
		{\cal P}(y)=\sup_{\tau \in {\cal S}} \mathbb{E}\left[\int_0^\tau e^{-\rho t}\left(\tilde{u}_B({\cal Y}_t^y )-\tilde{u}_A({\cal Y}_t^y)+{\cal Y}_t^y \e\right)dt\right]=
		\begin{cases}
		Dy^{n_2}+\Xi_h(y)\;\;\;&\mbox{for}\;\;y\ge z_R,\vspace{2mm}\\
		0\;\;\;&\mbox{for}\;\;\;0<y\le z_R,
		\end{cases}
		\end{split}
		\end{eqnarray}
where $z_R\in(0,\bar{z})$ is a unique solution of the following equation
\begin{equation*}
\int_{z_R}^\infty \nu^{-n_1-1}{h}(\nu)d\nu =0
\end{equation*}
and
\begin{equation*}
D=-\dfrac{2}{\t^2(n_1-n_2)}\int_0^{z_R}\nu^{-n_2-1}{h}(\nu)d\nu. 
\end{equation*}

Moreover, the optimal stopping time $\tau_R(y)$ is 
$$
\tau_R(y)=\inf\{t\ge 0 \mid {\cal Y}_t^y \le z_R \}.
$$
That is, 
\begin{equation*}
{\cal P}(y)=\sup_{\tau \in {\cal S}} \mathbb{E}\left[\int_0^\tau e^{-\rho t}\left(\tilde{u}_B({\cal Y}_t^y )-\tilde{u}_A({\cal Y}_t^y)+{\cal Y}_t^y \e\right)dt\right]=\mathbb{E}\left[\int_0^{\tau_R(y)} e^{-\rho t}\left(\tilde{u}_B({\cal Y}_t^y )-\tilde{u}_A({\cal Y}_t^y)+{\cal Y}_t^y \e\right)dt\right].
\end{equation*}
\end{pro}

We first derive the utility value of lifetime labor, ${\cal P}(y),$ in a heuristic and intuitive manner.  Suppose that there exists a boundary $z_R$ such that if the agent's marginal utility ${\cal Y}_t^y$ hits the boundary, the agent retires. The option value becomes zero after the agent retires. Hence, by the \textit{smooth pasting condition} we have 
	\be\label{boundary_condition}
	{\cal P}(z_R) = {\cal P}'(z_R)=0, \quad {\cal P}(y) = 0 \;\;\;\mbox{for}\;\; 0<y\le z_R.
	\ee
	
	When the agent is still working, ${\cal P}(y)$ satisfies  the Hamilton-Jacobi-Bellman(HJB) equation
	\be\label{eq:HJB}
	\dfrac{\t^2}{2}y^2{\cal P}''(y) + (\rho-r)y{\cal P}(y) - \rho {\cal P} + {h}(y)=0.
	\ee
	
	That is, for $y>z_R$, a general solution to the HJB equation \eqref{eq:HJB} can be represented as the sum of a general solution to the homogeneous equation and a particular solution as 
	\begin{equation*}
		{\cal P}(y) = D_1y^{n_1}+D_2y^{n_2} +\Xi_{h} (y),
	\end{equation*}
	where the particular solution $\Xi_h(y)$ is given by 
	\be\label{particular_solution}
	\Xi_h(y)=\dfrac{2}{\t^2(n_1-n_2)}\left[y^{n_2}\int_0^{y}\nu^{-n_2-1}{h}(\nu)d\nu +y^{n_1}\int_y^\infty \nu^{-n_1-1}{h}(\nu)d\nu\right]
	\ee
	To satisfy the transversality condition \eqref{eq:transversality} for ${\cal P}(y)$, the coefficient $D_1$ of $y^{n_1}$ should be zero, i.e, 
	$$
	D_1= 0. 
	$$
	
	Thus, we can write down ${\cal P}(y)$ as follows: 
	\begin{equation*}
		{\cal P}(y) = Dy^{n_2} +\Xi_h(y). 
	\end{equation*}

	By using the smooth pasting condition (${\cal P}(z_R)={\cal P}'(z_R)=0$), we have 
	\begin{equation*}
		\int_{z_R}^\infty \nu^{-n_1-1}{h}(\nu)d\nu =0\;\;\;\mbox{and}\;\;\;D=-\dfrac{2}{\t^2(n_1-n_2)}\int_0^{z_R}\nu^{-n_2-1}{h}(\nu)d\nu. 
	\end{equation*}
We will now proceed to give a formal proof of the proposition.

\begin{proof}

	Let us denote ${\cal G}(y)$ by 
		\begin{align*}
		{\cal G}(y)=\int_{y}^\infty \nu^{-n_1-1}{h}(\nu)d\nu=\int_{y}^\infty \nu^{-n_1}{\Psi}(\nu)d\nu.
		\end{align*}
		
	By Lemma \ref{lem:psi}, we deduce that ${\cal G}(y)$ is strictly increasing on $(0,\bar{z})$ and strictly decreasing on $(\bar{z},\infty)$. Moreover, 
	$$
	{\cal G}(z)>0\;\;\mbox{for all}\;\; z\ge \bar{z}.
	$$
	
	By Assumption \ref{as:limit-psi}, there exist $\d>0$ and $y_\d >0$ such that 
	$$
	\Psi(y)<-\d\;\;\;\mbox{for all}\;\;y\in(0,y_\d). 
	$$
	
	For sufficiently small $y<y_\d$, we have
	\begin{align*}
	{\cal G}(y) = \int_y^\infty \nu^{-n_1}\Psi(\nu)d\nu = &\int_y^{y_\d} \nu^{-n_1}\Psi(\nu)d\nu+\int_{y_\d}^\infty \nu^{-n_1}\Psi(\nu)d\nu\nonumber\\
	<&-\d\int_y^{y_\d}\nu^{-n_1}d\nu +\int_{y_\d}^\infty \nu^{-n_1}\Psi(\nu)d\nu\\
	=&-\d\dfrac{1}{1-n_1}({y_\d}^{1-n_1}-y^{1-n_1})+\int_{y_\d}^\infty \nu^{-n_1}\Psi(\nu)d\nu.\nonumber
	\end{align*}
	
	Since $\displaystyle\int_{y_\d}^\infty \nu^{-n_1}|\Psi(\nu)|d\nu<\infty$, 
	\begin{equation*}
	\lim_{y\to 0+}{\cal G}(y) < \d\dfrac{1}{n_1-1}({y_\d}^{1-n_1}-\lim_{y\to 0+} y^{1-n_1})+\int_{y_\d}^\infty \nu^{-n_1}\Psi(\nu)d\nu =-\infty. 
	\end{equation*}
	Thus, there exists a unique $z_R\in(0,\bar{z})$ such that ${\cal G}(z_R)=0$. \\
	
	\noindent {\bf Claim 1:} ${\cal P}(y)$, defined by 
	\begin{align*}
	{\cal P}(y)=
	\begin{cases}
	Dy^{n_2} + \Xi_h(y)\;\;\;&\mbox{for}\;\;y>z_R\vspace{2mm}\\
	0\;\;\;&\mbox{for}\;\;0<y\le z_R,
	\end{cases}
	\end{align*} 
	is continuously differentiable in $(0,\infty)$ and twice continuously differentiable in $(0,\infty)\backslash\{z_R\}$ and satisfies the variational inequality \eqref{eq:VI}. \\
	
	\noindent{\bf Proof of Claim 1:} By the construction of ${\cal P}(y)$, it is clear that ${\cal P}(y)$ is continuously differentiable in $(0,\infty)$ and twice continuously differentiable in $(0,\infty)\backslash\{z_R\}$.
	
	Note that for $y\in(0,z_R]$ 
	\begin{align*}
	\mathbb{L}{\cal P}(y)+h(y) = h(y) =y\Psi(y) < 0,
	\end{align*}
	where we have used that fact that $\Psi(y)$ is strictly increasing $y>0$ and $\Psi(\bar{z})=0$.
	
	Since $h(y)=\tilde{u}_B(y)-\tilde{u}_A(y)+\e y$, we have
	\begin{equation*}
	\Xi_h (y) = \Xi_{\tilde{u}_B}(y)-\Xi_{\tilde{u}_A}(y)+\dfrac{\e}{r}y. 
	\end{equation*}
	
	It follows from Lemma \ref{lem:Xi-Gamma} that 
	\begin{equation*}
	\Xi_h'(y) = \Gamma_{I_{u_A}}(y)-\Gamma_{I_{u_B}}(y) +\dfrac{\e}{r}
	\end{equation*}
	and 
	\begin{equation*}
	\lim_{y\uparrow \infty}\Xi_h'(y)=\lim_{y\uparrow \infty} \Gamma_{I_{u_A}}(y)- \lim_{y\uparrow \infty}\Gamma_{I_{u_B}}(y)+\dfrac{\e}{r}=\dfrac{\e}{r}.
	\end{equation*}
	
	Since ${\cal P}(y)=Dy^{n_2} + \Xi_h(y)$ for $y>z_R$, we have 
		\begin{align*}
		\lim_{y\uparrow \infty}{\cal P}'(y)=\lim_{y\uparrow \infty}(n_2Dy^{n_2-1}) + \lim_{y\uparrow \infty}\Xi_h'(y)=\dfrac{\e}{r}>0.
		\end{align*}
	
	Note that for $y>z_R$, 
		\begin{align*}
		{\cal P}(y)&=Dy^{n_2} + \Xi_h(y)\nonumber\\
		&=-\dfrac{2}{\t^2(n_1-n_2)}y^{n_2}\int_0^{z_R}\nu^{-n_2-1}{h}(\nu)d\nu+\dfrac{2}{\t^2(n_1-n_2)}\left[y^{n_2}\int_0^{y}\nu^{-n_2-1}{h}(\nu)d\nu +y^{n_1}\int_y^\infty \nu^{-n_1-1}{h}(\nu)d\nu\right]\\
		&=\dfrac{2}{\t^2(n_1-n_2)}\left[y^{n_2}\int_{z_R}^{y}\nu^{-n_2-1}{h}(\nu)d\nu +y^{n_1}\int_y^\infty \nu^{-n_1-1}{h}(\nu)d\nu\right].\nonumber
		\end{align*}

   Thus, 
   \begin{footnotesize}
   	 \begin{equation*}
   	{\cal P}'(y)=\dfrac{2}{\t^2(n_1-n_2)}\left[n_2y^{n_2-1}\int_{z_R}^{y}\nu^{-n_2-1}{h}(\nu)d\nu +n_1y^{n_1-1}\int_y^\infty \nu^{-n_1-1}{h}(\nu)d\nu\right]\;\;\mbox{for}\;\;y>z_R.
   	\end{equation*}
   \end{footnotesize}

Since $h(y)\ge 0$ for $y\ge \bar{z}$ and $h(y)<0$ for $y<\bar{z}$, we deduce that ${\cal P}'(y)$ is strictly increasing in $y\in(z_R,\bar{z})$ and strictly decreasing in $y>\bar{z}$. 

It follows from ${\cal P}'(z_R)=0$ and $\lim_{y\uparrow \infty}{\cal P}'(y)=\frac{\e}{r}$ that 
$$
{\cal P}'(y)>0\;\;\mbox{for}\;\;y>z_R. 
$$
That is, ${\cal P}(y)$ is strictly increasing in $y\in(z_R,\infty)$. Since ${\cal P}(z_R)=0$, we have 
$$
{\cal P}(y)>0\;\;\mbox{for}\;\;y>z_R. 
$$	

Therefore, ${\cal P}(y)$ satisfies the variational inequality \eqref{eq:VI}. This completes the proof of {\bf Claim 1. }\\

\noindent{\bf Claim 2:} There exists positive constant $\zeta$ such that 
\begin{equation*}
|{\cal P}'(y)|\le \zeta \left(y^{n_1-1}+y^{n_2-1}\right)
\end{equation*}
and $\lim_{y\to\infty} e^{-\rho t} \mathbb{E}\left[{\cal P}({\cal Y}_t^y)\right]=0$. \\

\noindent {\bf Proof of Claim 2:} Since
\begin{equation*}
{\cal P}(y) = \left(Dy^{n_2} +\Xi_h(y)\right){\bf 1}_{\{y>z_R\}}=\left(Dy^{n_2} +\Xi_{\tilde{u}_B}(y)-\Xi_{\tilde{u}_A}(y)+\dfrac{\e}{r}y\right){\bf 1}_{\{y>z_R\}},
\end{equation*}
we have
\begin{align*}
|{\cal P}'(y)|\le & \left|n_2 Dy^{n_2-1} +\Xi_{\tilde{u}_B}'(y)-\Xi_{\tilde{u}_A}'(y)+\dfrac{\e}{r} \right|\\
\le & n_2|D|y^{n_2-1}+|\Xi_{\tilde{u}_B}'(y)|+|\Xi_{\tilde{u}_A}'(y)|+\dfrac{\e}{r}.\nonumber
\end{align*}

Since $u_A, u_B\in\Sigma$, it follows from Lemma \ref{lem:finite} and Proposition \ref{pro:review:KMZ} (d) that 
\begin{equation*}
|{\cal P}'(y)|\le \zeta \left(y^{n_1-1}+y^{n_2-1}\right)
\end{equation*}
for some constant $\zeta>0$. 

	
	Moreover, 
	\begin{align*}
	|{\cal P}(y)| =& |\left(Dy^{n_2} +\Xi_h(y)\right){\bf 1}_{\{y>z_R\}}|\nonumber\\
	\le&|D|y^{n_2}{\bf 1}_{\{y>z_R\}} +|\Xi_{\tilde{u}_A}(y)|+|\Xi_{\tilde{u}_B}(y)|+\dfrac{\e}{r}y.\\
	\le&|D|z_R^{n_2} +|\Xi_{\tilde{u}_A}(y)|+|\Xi_{\tilde{u}_B}(y)|+\dfrac{\e}{r}y.\nonumber
	\end{align*}
	
	By Proposition \ref{pro:review:KMZ} (e), we have 
	\begin{equation*}
		\lim_{t\to\infty}e^{-\rho t}\mathbb{E}\left[|\Xi_{\tilde{u}_A}({\cal Y}_t^y)|\right]=	\lim_{t\to\infty}e^{-\rho t}\mathbb{E}\left[|\Xi_{\tilde{u}_B}({\cal Y}_t^y)|\right]=0.
	\end{equation*}
	
	Since $\lim_{t\to\infty} e^{-\rho t}|D|z_R^{n_2}=0$ and $\lim_{t\to\infty}e^{-\rho t}\mathbb{E}\left[{\cal Y}_t^y\right]=\lim_{t\to\infty} e^{-rt}y=0$, it follows that 
	\begin{equation*}
	\lim_{t\to\infty}e^{-\rho t}\mathbb{E}\left[|{\cal P}({\cal Y}_t^y)|\right]=0.
	\end{equation*}
	This completes the proof of {\bf Claim 2.}\\
	
	\noindent By {\bf Claim 1}, {\bf Claim 2} and Theorem \ref{thm:verification-OSP},  we have proved the desired results.
\end{proof}

Note that 
\begin{equation*}
\Xi_h(y) = \Xi_{\tilde{u}_B}(y)-\Xi_{\tilde{u}_A}(y)+\dfrac{\e}{r}y\;\;\mbox{and}\;\;J_A(y)=\Xi_{\tilde{u}_A}(y).
\end{equation*}

From \eqref{eq:dual_value}, we derive the following corollary.
\begin{cor}\label{cor:dual}
	The dual value function $J(y)$ is given by 
	\begin{eqnarray}\label{eq:form_J}
	\begin{split}
	J(y)=
	\begin{cases}
	Dy^{n_2}+\Xi_{\tilde{u}_B}(y)+\dfrac{\e}{r}y\;\;\;&\mbox{for}\;\;y\ge z_R,\vspace{2mm}\\
\Xi_{\tilde{u}_A}(y)\;\;\;&\mbox{for}\;\;\;0<y\le z_R.
	\end{cases}
	\end{split}
	\end{eqnarray}
\end{cor}

\section{ Optimal Policies and Human Wealth}\label{sec:optimal_policies}

In this section we derive the optimal policies of the agent and discuss their properties. 

\begin{lem}\label{lem:wealth-y}
	Let us denote ${\cal X}(y)$ and ${\Pi}(y)$ by 
	\begin{equation*}
	{\cal X}(y) = -J'(y)\;\;\;\mbox{and}\;\;\;\Pi(y)=\dfrac{\t}{\s}yJ''(y), 
	\end{equation*}
	respectively.
	
	Then, the dynamics of ${\cal X}({\cal Y}_t^y)$ follows 
	\begin{equation*}
	d{\cal X}({\cal Y}_t^{y})=\left( r{\cal X}({\cal Y}_t^{y})+(\m-r)\Pi({\cal Y}_t^y)-\hat{c}({\cal Y}_t^{y})+ \e{\bf 1}_{\{t<\tau_R(y)\}}\right)dt+\s \Pi({\cal Y}_t^y) dB_t
	\end{equation*}
	and 
	\begin{equation*}
	{\cal X}(y)= \mathbb{E}\left[\int_0^{\infty}\xi_t(\hat{c}({\cal Y}_t^y) -\e {\bf 1}_{\{t<\tau_R(y)\}})dt\right].
	\end{equation*}
\end{lem}
\begin{proof}
	Since $J(y)$ is smooth in $y>z_R$, it follows that 
	$$
	\dfrac{\theta^2}{2}y^2{\cal X}''(y) + (\rho-r+\theta^2)y{\cal X}'(y) - r{\cal X}(y)+I_{u_B}(y)-\e=0\;\;\mbox{for}\;\;y>z_R.
	$$
	
	By Proposition \ref{pro:review:KMZ}, the explicit-form of $J_A(y)$ is given by 
	\begin{equation*}
	J_A(y)=\dfrac{2}{\t^2(n_1-n_2)}\left[y^{n_2}\int_0^y \nu^{-n_2-1}\tilde{u}_A(\nu)d\nu +y^{n_1}\int_y^\infty \nu^{-n_1-1}\tilde{u}_A(\nu)d\nu\right].
	\end{equation*}
	
	Since $J_A(y)=\Xi_{\tilde{u}_A}(y)$ , Lemma \ref{lem:Xi-Gamma} implies that $J_A(y)$ is strictly convex in $y>0$ and 
	$$
	\lim_{y\uparrow \infty}J_A'(y)=0\;\;\mbox{and}\;\;\lim_{y\downarrow 0+}J_A'(y)=-\infty.
	$$
	
	Moreover, $J_A(y)$ satisfies the following ordinary differential equation(ODE):
	\begin{equation}\label{eq:J_A-ODE}
	\dfrac{\t^2}{2}y^2 J_A''(y)+(\rho-r)yJ_A'(y)-\rho J_A + \tilde{u}_A(y)=0. 
	\end{equation}
	
	For $0<y\le z_R$,  it follows from the ODE \eqref{eq:J_A-ODE} that 
	$$
	\dfrac{\theta^2}{2}y^2{\cal X}''(y) + (\rho-r+\theta^2)y{\cal X}'(y) - r{\cal X}(y)+I_{u_A}(y)=0\;\;\mbox{for}\;\;0<y\le z_R.
	$$
	
	Since 
	\begin{eqnarray}
	\begin{split}\label{eq:X}
	{\cal X}(y)=
	\begin{cases}
	-Dn_2y^{n_2-1}-\Xi_{\tilde{u}_B}'(y)-\dfrac{\e}{r}\;\;\;&\mbox{for}\;\;y\ge z_R,\vspace{2mm}\\
	-\Xi_{\tilde{u}_A}'(y)\;\;\;&\mbox{for}\;\;\;0<y\le z_R,
	\end{cases}
	\end{split}
	\end{eqnarray}
	it is easy to show that 
	\begin{equation*}
	{\cal X}(\cdot)\in {\cal W}_{loc}^{2,p}(0,\infty)\;\;\mbox{for any}\;\;p\ge 1,
	\end{equation*}
	where $\mathcal{W}^{2,p}(0,\infty),\,p\geq1,$ is the
	completion of $C^\infty(0,\infty)$ under the norm 
	$$
	\|\,V\,\|\,_{\mathcal{W}^{2,p}(0,\infty)}\triangleq \left[\;\int_{0}^{\infty}\,
	\left(\,|\,V\,|\,^p+|\,\partial_y V\,|\,^p
	+|\,\partial_{yy} V\,|\,^p\;\right)\,dy\,\right]^{1\over p}.
	$$
	and $\mathcal{W}^{2,p}_{{loc}}(0,\infty),\,p\geq1,$ is
	the set of all functions whose restrictions to the domain $\mathcal{K}$ belong to $\mathcal{W}^{2,p}(\mathcal{K})$ for any
	compact subset $\mathcal{K}$ of $(0,\infty)$.

	For a given $T>0$, the generalized It\^{o}'s lemma for $W_{loc}^{2,p}$ (see \citet{Krylov}) to ${\cal X}({\cal Y}_t^{y})$ yields that 
	\begin{eqnarray}
	\begin{split}\label{eq:dynamics:X1}
	d{\cal X}({\cal Y}_t^{y})=&{\cal X}'({\cal Y}_t^{y})d{\cal Y}_t^{y} +\dfrac{1}{2}{\cal X}''({\cal Y}_t^{y})(d{\cal Y}_t^{y})^2\\
	=&\left(\dfrac{\t^2}{2}({\cal Y}_t^{y})^2{\cal X}''({\cal Y}_t^{y})+(\rho-r){\cal Y}_t^{y} {\cal X}'({\cal Y}_t^{y})\right)dt-\t {\cal Y}_t^{y} {\cal X}'({\cal Y}_t^{y})dB_t\\
	=&\left( r{\cal X}({\cal Y}_t^{y})-\hat{c}({\cal Y}_t^{y})+\e{\bf 1}_{\{t<\tau_R(y)\}}+(\m-r)\left(-\frac{\t}{\s}{\cal X}'({\cal Y}_t^{y})\right) \right)dt+\s \left(-\frac{\t}{\s}{\cal Y}_t^{y}{\cal X}'({\cal Y}_t^{y})\right) dB_t\\
	=&\left( r{\cal X}({\cal Y}_t^{y})-\hat{c}({\cal Y}_t^{y})+ \e{\bf 1}_{\{t<\tau_R(y)\}}+(\m-r)\left(-\frac{\t}{\s}{\cal X}'({\cal Y}_t^{y})\right) \right)dt+\s \left(-\frac{\t}{\s}{\cal Y}_t^{y}{\cal X}'({\cal Y}_t^{y})\right) dB_t\\
	=&\left( r{\cal X}({\cal Y}_t^{y})-\hat{c}({\cal Y}_t^{y})+ \e{\bf 1}_{\{t<\tau_R(y)\}}\right)dt+\s \left(-\frac{\t}{\s}{\cal Y}_t^{y}{\cal X}'({\cal Y}_t^{y})\right) dB_t^\mathbb{Q},
	\end{split}
	\end{eqnarray}
	where the measure $\mathbb{Q}$ is defined in \eqref{eq:measure-Q}.
	
	Hence, we deduce that 
	\begin{footnotesize}
		\begin{eqnarray}
		\begin{split}\label{eq:dynamics:X2}
		d\left(e^{-rt}{\cal X}({\cal Y}_t^{y})\right)=e^{-rt}\left(\e{\bf 1}_{\{t<\tau_R(y)\}}-\hat{c}({\cal Y}_t^{y})\right)dt+e^{-rt}\s \left(-\frac{\t}{\s}{\cal Y}_t^y{\cal X}'({\cal Y}_t^{y})\right) dB_t^\mathbb{Q}.
		\end{split}
		\end{eqnarray}
	\end{footnotesize}
	
	From \eqref{eq:X}, we deduce that for any $y>0$ 
	\begin{align*}
	{\cal X}(y)\le& D|n_2| z_R^{n_2-1}+|\Gamma_{I_{u_A}}(y)| + |\Gamma_{I_{u_B}}(y)|+\dfrac{\e}{r},
	\end{align*}
	and
	\begin{align*}
	|y{\cal X}'(y)|\le& D|n_2(n_2-1)| z_R^{n_2-1}+|y\Gamma_{I_{u_A}}'(y)| + |y\Gamma_{I_{u_B}}'(y)|.
	\end{align*}
	
	Thus, it follows from Lemma \ref{lem:Xi-Gamma} (d) and (e) that 
	\begin{align*}
	\lim_{T\to\infty}e^{-rT}\mathbb{E}^{\mathbb{Q}}\left[{\cal X}({\cal Y}_T^{y}) \right]=\dfrac{1}{y}\lim_{T\to\infty}e^{-\rho T}\mathbb{E}\left[{\cal Y}_T^{y}{\cal X}({\cal Y}_T^{y}) \right]=0
	\end{align*}
	and the stopped process for $T$
	$$
	{\cal M}_{t\wedge T} =\int_0^{t\wedge T} e^{-r s }\t {\cal Y}_s^{y}{\cal X}'({\cal Y}_s^{y})dB^\mathbb{Q}_s
	$$
	is a martingale (see Lemma 3.4 in \citet{KMZ}). 
	
	By integrating the both sides of \eqref{eq:dynamics:X2} with respect to $t$, we have 
	\begin{footnotesize}
		\begin{eqnarray*}
		\begin{split}
		{\cal X}(y) =\mathbb{E}\left[\int_0^{T}{\xi_t} \left(\hat{c}({\cal Y}_t^{y})-\e){\bf 1}_{\{t<\tau_R(y)\}}\right)dt\right] +e^{-r T} \mathbb{E}^{\mathbb{Q}}\left[{\cal X}({\cal Y}_T^{y}) \right].
		\end{split}
		\end{eqnarray*}
	\end{footnotesize}
	
	Since $u_A, u_B \in \Sigma$, it follows from Proposition \ref{pro:review:KMZ} that 
	\begin{equation*}
	\mathbb{E}\left[\int_0^\infty \xi_t I_{\tilde{u}_A}({\cal Y}_t^{y})dt\right]<\infty\;\;\mbox{and}\;\;\mathbb{E}\left[\int_0^\infty \xi_t I_{\tilde{u}_B}({\cal Y}_t^{y})dt\right]<\infty.
	\end{equation*}
	
	Thus, letting $T\to\infty$, the {\it dominated convergence theorem} implies that
	\begin{footnotesize}
		\begin{eqnarray*}
		\begin{split}
		{\cal X}(y)=\mathbb{E}\left[\int_0^{\infty}{\xi_t} \left(\hat{c}({\cal Y}_t^{y})-\e{\bf 1}_{\{t<\tau_R(y)\}}\right)dt\right].
		\end{split}
		\end{eqnarray*}
	\end{footnotesize}

It follows from \eqref{eq:dynamics:X1} that 
	\begin{equation*}
d{\cal X}({\cal Y}_t^{y})=\left( r{\cal X}({\cal Y}_t^{y})+(\m-r)\Pi({\cal Y}_t^y)-\hat{c}({\cal Y}_t^{y})+ \e{\bf 1}_{\{t<\tau^*\}}\right)dt+\s \Pi({\cal Y}_t^y) dB_t\;\;\;\mbox{for}\;\;t\ge 0.
\end{equation*}
\end{proof}

We derive the agent's optimal policies by establishing the duality between the value function and the dual value function. 
\begin{theorem}~\label{thm:main} Given  $x>-\frac{\e}{r}$. 
	\begin{itemize} 
		\item[(a)] $V(x)$ and $J(y)$ satisfy the duality relationship:
		\begin{equation}\label{eq:duality}
			V(x) = \inf_{y>0} \left(J(y)  + y x\right), \quad 	J(y) = \sup_{x>-\frac{\e}{r}} \left(V(x)  - y x\right).
		\end{equation}
		There exists a unique $y^*>0$ such that 
		\begin{equation*}
			x = -J'(y^*). 
		\end{equation*}
		\item[(b)] The optimal policies $(c^*, \pi^*, \tau^*)$ are 
		\begin{footnotesize}
			\begin{eqnarray*}
				\begin{split}
					c_t^* =\hat{c}({\cal Y}_t^{y^*})=  
					\begin{cases}
						I_B({\cal Y}_t^{y^*})\;\;\;&\mbox{for}\;\;0\le t < \tau^*,\vspace{1mm}\\
						I_A({\cal Y}_t^{y^*})\;\;\;&\mbox{for}\;\;t \ge \tau^*,
					\end{cases},\;\;\;\tau^*=\tau_R(y^*)=\inf\{t\ge0\;\mid {\cal Y}_t^{y^*}\le z_R \},
				\end{split}
			\end{eqnarray*}
		\end{footnotesize}
		and 
		\begin{footnotesize}
			\begin{eqnarray*}
			\pi_t^*=\dfrac{\t}{\s}{\cal Y}_t^* J''({\cal Y}_t^{y^*})
			\end{eqnarray*}
		\end{footnotesize}
		where $X_t = - J^{\prime}({\cal Y}_t^{y^*})$ and ${\cal Y}_t^{y^*} = y^* e^{\rho t }{\xi}_t$.

	\end{itemize}
\end{theorem}

\begin{proof}
	Since $z_R<\bar{z}$ and $h(y)<0$ for $y\in(0,\bar{z})$, we have 
	\begin{equation*}
	D = -\dfrac{2}{\t^2(n_1-n_2)}\int_0^{z_R}\nu^{-n_2-1}{h}(\nu)d\nu>0. 
	\end{equation*}

	
By \eqref{eq:form_J},	it follows from Lemma \ref{lem:Xi-Gamma} and $D>0$ that 
	\begin{equation*}
	J''(y) = n_2(n_2-1)Dy^{n_2-2} +\Xi_{\tilde{u}_B}''(y) >0\;\;\mbox{for}\;\;y>z_R.
	\end{equation*}
	
	Since $J(y)=J_A(y)=\Xi_{\tilde{u}_A}(y)$ for $0<y\le z_R$, it is clear that 
	\begin{equation*}
	J''(y) =\Xi_{\tilde{u}_A}''(y) >0\;\;\mbox{for}\;\;0<y\le z_R. 
	\end{equation*}
	Hence, $J(y)$ is strictly convex in $y>0$. 
	
	Note that 
	\begin{equation*}
	\lim_{y\uparrow \infty}{\cal P}'(y)=\dfrac{\e}{r},\;\;\lim_{y\downarrow 0}{\cal P}'(y)=0
	\end{equation*}
	and 
	\begin{equation*}
	\lim_{y\uparrow \infty}J_A'(y)=0,\;\;\lim_{y\downarrow 0}J_A'(y)=-\infty.
	\end{equation*}
	
	Thus, it follows that 
	\begin{align*}
	&\lim_{y\uparrow \infty} {J}'(y)=\dfrac{\e}{r},\vspace{2mm}\\
	&\lim_{y\downarrow 0} {J}'(y)=-\infty.
	\end{align*}
	
	Therefore, there exists a unique $y^*>0$ such that for given $x>-\frac{\e}{r}$
	$$
	x = -J'(y^*). 
	$$

	It follows from Lemma \ref{lem:wealth-y} that 
	\begin{equation*}
	x=-J'(y^*)={\cal X}(y^*)=\mathbb{E}\left[\int_0^{\infty}{\xi_t}\left(\hat{c}({\cal Y}_t^{y^*})-\e{\bf 1}_{\{t<\tau^*\}}\right)dt\right].
	\end{equation*}
    Hence, we have 
		\begin{eqnarray*}
		\begin{split}
		y^*x=&y^*\mathbb{E}\left[\int_0^{\infty}{\xi_t}\left(\hat{c}({\cal Y}_t^{y^*})-\e{\bf 1}_{\{t<\tau^*\}}\right)dt\right]\\
		=&\mathbb{E}\left[\int_0^{\infty}{\cal Y}_t^{y^*}\left(\hat{c}({\cal Y}_t^{y^*})-\e{\bf 1}_{\{t<\tau^*\}}\right)dt\right]\\
		=&\mathbb{E}\left[\int_0^{\tau^*} e^{-\rho t}u_A(\hat{c}({\cal Y}_t^{y^*}))dt+\int_{\tau^*}^\infty e^{-\rho t}u_B(\hat{c}({\cal Y}_t^{y^*}))dt\right]\\-&\mathbb{E}\left[\int_0^{\tau^*} e^{-\rho t}\left(\tilde{u}_B({\cal Y}_t^{y^*} )+{\cal Y}_t^{y^*}\e\right)dt+\int_0^{\tau^*} e^{-\rho t}\tilde{u}_A({\cal Y}_t^{y^*} )dt\right]\\
		=&\mathbb{E}\left[\int_0^{\tau^*} e^{-\rho t}u_A(\hat{c}({\cal Y}_t^{y^*}))dt+\int_{\tau^*}^\infty e^{-\rho t}u_B(\hat{c}({\cal Y}_t^{y^*}))dt\right]-J(y^*).
		\end{split}
		\end{eqnarray*}
	
	Hence,
	\begin{footnotesize}
		\begin{align*}
		J(y^*) +y^*x=&\mathbb{E}\left[\int_0^{\tau^*} e^{-\rho t}u_A(\hat{c}({\cal Y}_t^{y^*}))dt+\int_{\tau^*}^\infty e^{-\rho t}u_B(\hat{c}({\cal Y}_t^{y^*}))dt\right]\\
		\le &\sup_{(c,\pi,\tau\in{\cal A}(x)} \mathbb{E}\left[\int_0^\tau e^{-\rho t} u_A(c_t)dt+\int_\tau^\infty e^{-\rho t}u_B(c_t)dt\right]. 
		\end{align*}
	\end{footnotesize}
	
	For any $y>0$ and $(c,\pi,\tau)\in {\cal A}(x)$,  
	\begin{footnotesize}
		\begin{align*}
		&\mathbb{E}\left[\int_0^\tau e^{-\rho t} u_A(c_t)dt+\int_\tau^\infty e^{-\rho t}u_B(c_t)dt\right]\\
		\le&\mathbb{E}\left[\int_0^\tau e^{-\rho t} u_A(c_t)dt+e^{-\rho \tau}V_A(X_\tau)\right]+y\left(x-\mathbb{E}\left[\int_0^\tau {\xi_t} (c_t - \e)dt + \xi_{\tau}X_{\tau} \right]\right)\\
		\le&\mathbb{E}\left[\int_0^\tau e^{-\rho t}\left(\tilde{u}_A({\cal Y}_t^y)+{\cal Y}_t^y \e\right)dt +e^{-\rho \tau}J_A({\cal Y}_{\tau}^y)\right]+yx. 
		\end{align*}
	\end{footnotesize}
	
	This implies that 
	\begin{align*}
	&\sup_{(c,\pi,\tau)\in{\cal A}(x)} \mathbb{E}\left[\int_0^\tau e^{-\rho t} u_A(c_t)dt+\int_\tau^\infty e^{-\rho t}u_B(c_t)dt\right] \\
	\le&\sup_{\tau\in{\cal S}}\inf_{y>0}\left(\mathbb{E}\left[\int_0^\tau e^{-\rho t}\left(\tilde{u}_A({\cal Y}_t^y)+{\cal Y}_t^y \e\right)dt +e^{-\rho \tau}J_A({\cal Y}_{\tau}^y)\right]+yx\right)\\
	\le&\inf_{y>0}\sup_{\tau\in{\cal S}}\left(\mathbb{E}\left[\int_0^\tau e^{-\rho t}\left(\tilde{u}_A({\cal Y}_t^y)+{\cal Y}_t^y \e\right)dt +e^{-\rho \tau}J_A({\cal Y}_{\tau}^y)\right]+yx\right)=\inf_{y>0}\left(J(y)+yx\right).
	\end{align*}
	
	Overall, we have 
	\begin{align*}
	J(y^*) +y^*x\le \inf_{y>0}\left(J(y)+yx\right) \le J(y^*) +y^*x
	\end{align*}
	and  
	\begin{footnotesize}
		\begin{align*}
		\mathbb{E}\left[\int_0^{\tau^*} e^{-\rho t}u_A(\hat{c}({\cal Y}_t^{y^*}))dt+\int_{\tau^*}^\infty e^{-\rho t}u_B(\hat{c}({\cal Y}_t^{y^*}))dt\right]=\sup_{(c,\pi,\tau)\in{\cal A}(x)}\mathbb{E}\left[\int_0^\tau e^{-\rho t} u_A(c_t)dt+\int_\tau^\infty e^{-\rho t}u_B(c_t)dt\right].
		\end{align*}
	\end{footnotesize}
	
	That is, $c_t^*=\hat{c}({\cal Y}_t^{y^*})$ and $\tau^*=\tau^*(y^*)$ are optimal. 
	
	By Proposition \ref{pro:budget-equaltiy}, there exist a portfolio $\pi_t^*$ such that $(c^*,\pi^*,\tau^*)\in{\cal A}(x)$ and the corresponding wealth process $X^{x, c^*, \pi^*}$ is 
		\begin{equation}\label{eq:compare1}
		dX_t^{x,c^*,\pi^*} = [r X_t^{x,c^*,\pi^*} +(\m-r)\pi_t^* -c_t^* +\e{\bf 1}_{\{t<\tau^*\}}]dt +\s \pi_t^* dB_t \;\;\;t\ge 0
		\end{equation}
		and 
		\begin{equation*}
		X_t^{x,c^*,\pi^*} =\mathbb{E}_t\left[\int_t^\infty \dfrac{ \xi_s}{ \xi_t} (c_s^*-\e 1_{\{s<\tau^* \}})ds\right].
		\end{equation*} 
	The {\it strong Markov property} implies that 
    \begin{equation*}
    \xi_t{\cal X}({\cal Y}_t^{y^*})=\mathbb{E}\left[\int_t^{\infty}\xi_s\left(\hat{c}({\cal Y}_s^{y^*})-\e{\bf 1}_{\{s<\tau_R({\cal Y}_t^{y^*})\}}\right)dt\right]=\xi_t X_t^{x,c^*,\pi^*}.
    \end{equation*}
    
    That is, 
    \begin{equation}\label{eq:compare2}
    {\cal X}({\cal Y}_t^{y^*})=X_t^{x,c^*,\pi^*}\;\;\;\mbox{for}\;\;t\ge0.
    \end{equation}
    
    By Proposition \ref{pro:budget-equaltiy}, 
    \begin{equation}\label{eq:compare3}
    	d{\cal X}({\cal Y}_t^{y^*})=\left( r{\cal X}({\cal Y}_t^{y^*})+(\m-r)\Pi({\cal Y}_t^{y^*})-\hat{c}({\cal Y}_t^{y^*})+ \e{\bf 1}_{\{t<\tau_R(y^*)\}}\right)dt+\s \Pi({\cal Y}_t^{y^*}) dB_t
    \end{equation}
    
    It follows from \eqref{eq:compare1}, \eqref{eq:compare2}, and \eqref{eq:compare3} that 
    $$
    \pi_t^*=\Pi({\cal Y}_t^{y^*})=\frac{\t}{\s}{\cal Y}_t^{y^*}J''({\cal Y}_t^{y^*}),
    $$

\end{proof}

 Lemma \ref{lem:wealth-y} and Theorem \ref{thm:main} imply
$$
X_t = -J^{\prime}({\cal Y}_t^{y^*})=\mathbb{E}\left[\int_t^{\infty} e^{-\rho (s-t) }\frac{\xi_s}{\xi_t} \left(\hat{c}({\cal Y}_s^{y^*})-\e{\bf 1}_{\{s<\tau_R(y^*)\}}  \right)ds\right].
$$

Thus, we establish the following relationship:
\be\label{eq:total_wealth}
 \mathbb{E}\left[\int_t^\infty e^{-\rho (s-t) }\frac{\xi_s}{\xi_t} c_s^*ds\right] =X_t + \mathbb{E}\left[\int_t^{\tau^*} e^{-\rho (s-t) }\frac{\xi_s}{\xi_t} \e ds\right].
\ee
The left-hand side is the present value of lifetime consumption. The right-hand side consists of two components, financial wealth $X_t$ and the present value of labor income, which can be regarded as {\em human wealth}.  Thus, the right-hand side is equal to the agent's {\em total wealth}, comprised of financial wealth and human wealth. Equation \eqref{eq:total_wealth} says that life time consumption is financed by financial wealth and human wealth.

Theorem \ref{thm:main} also implies that the agent retires when the marginal utility of wealth reaches threshold level $z_R$, or equivalently, the agent's wealth level reaches threshold  $x_R$, satisfying
\begin{equation}\label{eq:retirement-x}
x_R = -J'(z_R)=-J_A'(z_R),
\end{equation}
where we have used ${\cal P}'(z_R)=0$, i.e., human wealth is equal to 0 at retirement. 

\begin{pro}\label{pro:wealth-boundary}
	If $\e$ increases, then the optimal retirement threshold $x_R$ increases.
\end{pro}
\begin{proof}
	
Let $\Psi_i(y)$ be the current normalized marginal benefit of work for given $\e_i>0$ $(i=1,2)$. 

Suppose that $\e_1>\e_2$. Then, 
$$
\Psi_1(y)>\Psi_2(y) \;\;\;\mbox{for any}\;\;y>0.
$$

Note that there exists a unique $\bar{z}_i>0$ such that $\Psi_i(\bar{z}_i)=0$ for $i=1,2$.

Since $0=\Psi_1(\bar{z}_1)=\Psi_2(\bar{z}_2)>\Psi_2(\bar{z}_1)$, it follows from Lemma \ref{lem:psi} that
$$
\bar{z}_2 >\bar{z}_1. 
$$ 

For $i=1,2,$ there exists a unique $z_{R,i}\in(0,\bar{z}_i)$ such that 
\begin{equation*}
{\cal G}_i(z_{R,i})=0,
\end{equation*}
where 
$$
{\cal G}_i(y)\equiv \int_y^\infty \nu^{-n_1}\Psi_i(\nu)d\nu.
$$

It follows from $\Psi_1(y)>\Psi_2(y)$ for all $y>0$ that 
\begin{equation*}
0={\cal G}_1(z_{R,1})={\cal G}_2(z_{R,2})>{\cal G}_2(z_{R,1}).
\end{equation*}

Since ${\cal G}_2(y)$ is strictly increasing in $y\in(0,\bar{z}_2)$, it follows from $\bar{z}_2>\bar{z}_1>z_{R,1}$ that 
\begin{equation*}
z_{R,2}>z_{R,1}. 
\end{equation*}
That is, the free boundary $z_R$ decreases as $\e$ increases. 

Let $x_{R,i}$ be the optimal retirement threshold corresponding to $z_{R,i}$. 

Since $J_A(y)$ is strictly convex in $y>0$, it follows from $x_{R,i}=-J_A'(z_{R,i})$ that 
\begin{equation*}
x_{R,1}>x_{R,2}. 
\end{equation*}

Therefore, if $\e>0$ increases, the optimal retirement threshold $x_R$ increases.

\end{proof}

\begin{pro}
	If $\m-r>0\; (\m-r<0)$, the optimal portfolio $\pi^*$ jumps downward\;(upward) just after retirement. In other words, 
	\begin{equation*}
	\pi_{\tau^*-}-\pi_{\tau^*+}=\lim_{y\to z_R+}\Pi(y)-\lim_{y\to z_R-}\Pi(y)=-\dfrac{2}{\m-r}\Psi(z_R).
	\end{equation*}
\end{pro}
\begin{proof}
Since $J(y)=J_A(y)+{\cal P}(y)$, it is easy to obtain that 
\begin{align}\label{eq:jump1}
	\lim_{y\to z_R+}\Pi(y)-\lim_{y\to z_R-}\Pi(y)&=\dfrac{\t}{\s}z_R{\cal P}''(z_R).
\end{align}
Note that for $y>z_R$, ${\cal P}(y)$ satisfies 
\begin{equation*}
\dfrac{\t^2}{2}y^2{\cal P}''(y) + (\rho-r)y{\cal P}(y) - \rho {\cal P} + {h}(y)=0.
\end{equation*}

Since ${\cal P}(z_R)={\cal P}'(z_R)=0$, we have 
\begin{align}\label{eq:jump2}
\lim_{y\to z_R+}{\cal P}''(y)=&-\lim_{y\to z_R+}\dfrac{2}{\t^2}\dfrac{1}{y^2}\left((\rho-r)y{\cal P}(y) - \rho {\cal P} + {h}(y)\right)
=-\dfrac{2}{\t^2}\dfrac{1}{z_R}h(z_R)=-\dfrac{2}{\t^2}\dfrac{1}{z_R}\Psi(z_R).
\end{align}

By \eqref{eq:jump1} and \eqref{eq:jump2}, we deduce that 
	\begin{equation*}
\lim_{y\to z_R+}\Pi(y)-\lim_{y\to z_R-}\Pi(y)=-\dfrac{2}{\m-r}\Psi(z_R).
\end{equation*}
\end{proof}

\section{Example}\label{sec:examples}

In this section we provide an example.

For $u\in \Sigma$, let us consider the following utility functions:
\be\label{ex}
u_B(c)=u(c)-l\;\;\;\mbox{and}\;\;u_A(c)=u(kc+b),\;\; l \geq 0, k\geq 1, b\geq 0, \;\;\mbox{with}\;\;{(k-1)^2+l^2\ne 0}.
\ee

The case $l>0, k=1, b=0$ is the model of \citet{CS2006}, the case $l=0, k>1, b=0$ is that of \citet{FP2007} and \citet{DL2010}, and the case $l>0, b>0$ is that of  \citet{BJKP2020}. Thus, the example encompasses major models in the literature\footnote{{\citet{FP2007}, \citet{DL2010} and \citet{BJKP2020} only consider the constant relative risk aversion(CRRA) felicity function.}}.
{
\begin{pro}\label{pro:assumption}
	$u_A$ and $u_B$ in \eqref{ex} satisfy Assumptions \ref{as:utility}-\ref{as:limit-psi}. 
\end{pro}
\begin{proof}
	We have
	\begin{align*}
	\tilde{u}_B(y)&=\sup_{c\ge 0}\left(u(c)-l-yc\right)=u(I_u(y))-yI_u(y)-l=\tilde{u}(y)-l,
	\end{align*}
	\begin{align*}
	\tilde{u}_A(y)=\sup_{c\ge 0}\left(u(kc+b)-yc\right)=&\left(u(I_u(\dfrac{y}{k}))-\dfrac{y}{k}\left(I_u(\dfrac{y}{k})-b\right)\right){\bf 1}_{\{y/k\le u'(b) \}} +u(b){\bf 1}_{\{y/k>u'(b)\}}\\
	=&\left(\tilde{u}(\dfrac{y}{k})+b\dfrac{y}{k}\right){\bf 1}_{\{y/k\le u'(b) \}} +u(b){\bf 1}_{\{y/k>u'(b)\}}
	\end{align*}
	and 
	\begin{align}
	I_{u_B}(y)=I_u(y)\;\;\;\mbox{and}\;\;\;I_{u_A}(y)=\dfrac{1}{k}\left(I_u(\dfrac{y}{k})-b\right){\bf 1}_{\{y/k\le u'(b) \}}.
	\end{align}
	
	Since $u\in \Sigma$, it is easy to check that 
	\begin{equation}\label{eq:assumption1}
	u_A,\;u_B \in \Sigma. 
	\end{equation}
	
	Note that 
	\begin{align}\label{eq:assumption2}
	u_A(I_A(y))-u_B(I_A(y))=&\;u\left(I_u(\dfrac{y}{k}){\bf 1}_{\{y/k\le u'(b) \}}+b{\bf 1}_{\{y/k> u'(b) \}}\right)-\left(u(I_u(y))-l\right)\nonumber\\
	\ge&\;u(I_u(\dfrac{y}{k}))-u(I_u(y))+l\\
	>&\;0, \nonumber
	\end{align}
	where we have used the fact $(k-1)^2+l^2 \ne 0$ in the second inequality.
	\medskip
	Since $\tilde{u}_A(y)=\sup_{c\ge 0}\left(u(kc+b)-yc\right)\ge \sup_{c\ge 0}\left(u(kc)-yc\right)=\tilde{u}(\dfrac{y}{k})$, 
	we deduce that 
	\begin{align}
	\Psi(y)=&\dfrac{1}{y}\left(\tilde{u}_B(y)-\tilde{u}_A(y)\right)+\e\\
	\le&\dfrac{1}{y}\left(\tilde{u}(y)-l - \tilde{u}(\frac{y}{k})\right)+\e. \nonumber
	\end{align}
	
	If $k=1$ and $l>0$, then 
	\begin{equation}\label{eq:assumption3}
	\lim_{y\to 0+}\Psi(y) \le \lim_{y\to0+} \dfrac{1}{y}\left(\tilde{u}(y)-l - \tilde{u}(y)\right)+\e =-\infty.
	\end{equation}
	
	If $k\ne1$, then it follows from the mean-value theorem that for any $y>0$ there exists $y_\d\in (y,y/k)$ such that 
	\begin{equation*}
	\Psi(y)\le \frac{1}{y}(\tilde{u}(y)-l- \tilde{u}(\frac{y}{k}))+\e=-\left(1-\frac{1}{k}\right){I}_u(y_\d)-\dfrac{l}{y}+\e. 
	\end{equation*} 
	
	Since $\lim_{y\to 0+}I_u(y)=+\infty$, we deduce that 
	\begin{equation}\label{eq:assumption4}
	\lim_{y\to 0+}\Psi(y)\le\lim_{y\to 0+}\left[\frac{1}{y}(\tilde{u}(y)-l- \tilde{u}(\frac{y}{k}))+\e\right]=\lim_{y\to 0+}\left[-\left(1-\frac{1}{k}\right){I}_u(y_\d)-\dfrac{l}{y}+\e\right]=-\infty. 
	\end{equation} 
	
	From \eqref{eq:assumption1}, \eqref{eq:assumption2}, \eqref{eq:assumption3}, and \eqref{eq:assumption4}, we conclude that $u_A$ and $u_B$ in \eqref{ex} satisfy Assumptions \ref{as:utility}-\ref{as:limit-psi}. 
\end{proof}

By Proposition \ref{pro:assumption}, all results in Sections \ref{sec:optimization}-\ref{sec:optimal_policies} can be applied when $u_A$, $u_B$ are given by \eqref{ex}.


Recall that $z_R\in(0,\bar{z})$ is a unique solution satisfying
\begin{align*}
{\cal G}(y)=0,
\end{align*}
where $${\cal G}(y)=\int_{y}^\infty \nu^{-n_1-1}{h}(\nu)d\nu=\int_{y}^\infty \nu^{-n_1}{\Psi}(\nu)d\nu.$$

Moreover, ${\cal G}(y)$ is strictly increasing and decreasing in $y\in(0,\bar{z})$ and $y\in(\bar{z},\infty)$, respectively. 
Since $\lim_{y\to 0+}{\cal G}(y)=-\infty$ and ${\cal G}(y)>0\;\;\mbox{for}\;\;y\ge \bar{z}$, we can easily derive the following lemma:
\begin{lem}\label{lem:z_R-position}
	$z_R<ku'(b)$ if and only if ${\cal G}(k u'(b))>0$. 
\end{lem}

\subsection{CRRA felicity function}

We assume that the agent has constant relative risk aversion(CRRA), i.e.,
\begin{equation}\label{eq:CRRA}
u(c)=\dfrac{c^{1-\g}}{1-\g},\;\;\;\g\ne 1,\;\g>0,
\end{equation}	
where $\g$ denotes the coefficient of relative risk aversion.

Then, 
\begin{equation*}
u_B(c)=\dfrac{c^{1-\g}}{1-\g}-l\;\;\;\mbox{and}\;\;u_A(c)=\dfrac{(kc+b)^{1-\g}}{1-\g},
\end{equation*}
where $l \geq 0, k\geq 1, b\geq 0, \;\;\mbox{with}\;\;{(k-1)^2+l^2\ne 0}$. 

Note that, under the CRRA utility \eqref{eq:CRRA}, the integrabiity condition \eqref{as:well-defined} is equivalent to the following assumption:
\begin{as}\label{as:Merton} The Merton constant $M$ defined by 
\begin{equation}
M\equiv r+\dfrac{\rho-r}{\g} +\dfrac{\g-1}{\g^2}\dfrac{\t^2}{2}>0
\end{equation}	
is positive. 
\end{as}

We can easily obtain that
\begin{align*}
I_{u_B}(y)&=y^{-\frac{1}{\g}},\;\;\;I_{u_A}(y)=\dfrac{1}{k}\left(\left(\dfrac{y}{k}\right)^{-\frac{1}{\g}}-b\right){\bf 1}_{\{y\le kb^{-\g} \}},\vspace{2mm}\\
\tilde{u}_B(y)&=\dfrac{\g}{1-\g}y^{-\frac{1-\g}{\g}}-l,\;\;\tilde{u}_A(y)=\left(\dfrac{\g}{1-\g}\left(\dfrac{y}{k}\right)^{-\frac{1-\g}{\g}}+b\dfrac{y}{k}\right){\bf 1}_{\{y\le k b^{-\g} \}} +\dfrac{b^{1-\g}}{1-\g}{\bf 1}_{\{y>k b^{-\g}\}}.\nonumber
\end{align*}

By Corollary \ref{cor:dual}, the dual value function $J(y)$ is given by 
\begin{align*}
J(y)=
\begin{cases}
Dy^{n_2}+\Xi_{\tilde{u}_B}(y)+\dfrac{\e}{r}y\;\;\;&\mbox{for}\;\;y\ge z_R,\vspace{2mm}\\
\Xi_{\tilde{u}_A}(y)\;\;\;&\mbox{for}\;\;\;0<y\le z_R.
\end{cases}
\end{align*}

Since 
\begin{align*}
\Xi_{\tilde{u}_B}(y)&=\dfrac{2}{\t^2(n_1-n_2)}\left[y^{n_2}\int_0^y \nu^{-n_2-1}\tilde{u}_B(\nu)d\nu +y^{n_1}\int_y^\infty \nu^{-n_1-1}\tilde{u}_B(\nu)d\nu\right],\\
\Xi_{\tilde{u}_A}(y)&=\dfrac{2}{\t^2(n_1-n_2)}\left[y^{n_2}\int_0^y \nu^{-n_2-1}\tilde{u}_A(\nu)d\nu +y^{n_1}\int_y^\infty \nu^{-n_1-1}\tilde{u}_A(\nu)d\nu\right],
\end{align*}
we have 
\begin{align*}
\Xi_{\tilde{u}_B}(y)=\dfrac{2}{\t^2(n_1-n_2)}\left[y^{n_2}\int_0^y \nu^{-n_2-1}\tilde{u}_B(\nu)d\nu +y^{n_1}\int_y^\infty \nu^{-n_1-1}\tilde{u}_B(\nu)d\nu\right]=\dfrac{1}{M}\dfrac{\g}{1-\g}y^{-\frac{1-\g}{\g}}-\dfrac{l}{\rho},
\end{align*}
and
\begin{footnotesize}
	\begin{eqnarray}
	\begin{split}
	\Xi_{\tilde{u}_A}(y)=
	\begin{cases}
	\phi_1(y)	\;\;&\mbox{for}\;\;\;0<y\le kb^{-\g}, \vspace{1mm}\\
	\phi_2(y)\;\;&\mbox{for}\;\;y>kb^{-\g},
	\end{cases}
	\end{split}
	\end{eqnarray}
\end{footnotesize}
where $\phi_1(y)$ and $\phi_2(y)$ are given by 
\begin{footnotesize}
	\begin{align*}
	\phi_1(y)=&\left(\dfrac{\frac{1}{M}\left(\frac{\g n_2}{1-\g}+1\right)+\frac{n_2-1}{r}-\frac{n_2}{\rho(1-\g)}}{n_1-n_2}\right)b^{1-\g+\g n_1}\left(\dfrac{y}{k}\right)^{n_1}+\dfrac{1}{K}\dfrac{\g}{1-\g}(\dfrac{y}{k})^{-\frac{1-\g}{\g}}+\dfrac{b}{r}\dfrac{y}{k}
	\end{align*}
\end{footnotesize}
and 
\begin{footnotesize}
	\begin{align*}
	\phi_2(y)=& \left(\dfrac{\frac{1}{M}\left(\frac{\g n_1}{1-\g}+1\right)+\frac{n_1-1}{r}-\frac{n_1}{\rho(1-\g)}}{n_1-n_2}\right)b^{1-\g+\g n_2}\left(\dfrac{y}{k}\right)^{n_2}+\dfrac{b^{1-\g}}{\rho(1-\g)},
	\end{align*}
\end{footnotesize}
respectively. 

Since the current normalized marginal benefit of work $\Psi(y)$ is given by 
\begin{equation*}
\Psi(y)=\dfrac{1}{y}\left[\dfrac{\g}{1-\g}y^{-\frac{1-\g}{\g}}-l-\left(\dfrac{\g}{1-\g}\left(\dfrac{y}{k}\right)^{-\frac{1-\g}{\g}}+b\dfrac{y}{k}\right){\bf 1}_{\{y\le k b^{-\g} \}}-\dfrac{b^{1-\g}}{1-\g}{\bf 1}_{\{y>k b^{-\g}\}}\right]+\e,
\end{equation*}
we have
\begin{align*}
{\cal G}(kb^{-\g})=&\int_{kb^{-\g}}^\infty \nu^{-n_1}\Psi(\nu)d\nu\\
=&\int_{kb^{-\g}}^\infty \nu^{-n_1-1}\left(\dfrac{\g}{1-\g}\nu^{-\frac{1-\g}{\g}}-l-\dfrac{b^{1-\g}}{1-\g}+\e\nu\right)d\nu\\
=&\left[\dfrac{\g}{1-\g}\dfrac{1}{1-n_1-\frac{1}{\g}}\nu^{1-n_1-\frac{1}{\g}}+l\dfrac{\nu^{-n_1}}{n_1}+\dfrac{b^{1-\g}}{1-\g}\dfrac{\nu^{-n_1}}{n_1}-\e\dfrac{\nu^{1-n_1}}{n_1-1}\right]_{\nu=kb^{-\g}}^{\infty}\\
=&-\left[\dfrac{\g}{1-\g}\dfrac{1}{1-n_1-\frac{1}{\g}}(kb^{-\g})^{1-n_1-\frac{1}{\g}}+l\dfrac{(kb^{-\g})^{-n_1}}{n_1}+\dfrac{b^{1-\g}}{1-\g}\dfrac{(kb^{-\g})^{-n_1}}{n_1}-\e\dfrac{(kb^{-\g})^{1-n_1}}{n_1-1}\right].
\end{align*}

\noindent Lemma \ref{lem:z_R-position} implies that
\item[(Case 1)] ${\cal G}(kb^{-\g})>0$ (or equivalently, $z_R<kb^{-\g}$)\\

\noindent In this case, the coefficient $D$ is given by 
\begin{align*}
D=&-\dfrac{2}{\t^2(n_1-n_2)}\int_0^{z_R}\nu^{-n_2}\Psi(\nu)d\nu\\
 =&-\dfrac{2}{\t^2(n_1-n_2)}\int_0^{z_R}\nu^{-n_2-1}\left[\dfrac{\g}{1-\g}\n^{-\frac{1-\g}{\g}}-l-\left(\dfrac{\g}{1-\g}\left(\dfrac{\n}{k}\right)^{-\frac{1-\g}{\g}}+b\dfrac{\n}{k}\right)+\e\n\right]d\nu\\
 =&-\dfrac{2}{\t^2(n_1-n_2)}\left[\dfrac{\g}{1-\g}\dfrac{1}{1-n_2-\frac{1}{\g}}\left(1-k^{\frac{1-\g}{\g}}\right)(z_R)^{1-n_2-\frac{1}{\g}}+l\dfrac{(z_R)^{-n_2}}{n_2}-\dfrac{(z_R)^{1-n_2}}{n_2-1}\left(\e-\dfrac{b}{k}\right)\right]
\end{align*}
and $z_R$ is a unique solution of the following algebraic equation:
\begin{align*}
0={\cal G}(z_R)=&\int_{z_R}^\infty \nu^{-n_1}\Psi(\nu)d\nu\\
=&\int_{z_R}^{kb^{-\g}}\nu^{-n_1-1}\left[\dfrac{\g}{1-\g}\n^{-\frac{1-\g}{\g}}-l-\left(\dfrac{\g}{1-\g}\left(\dfrac{\n}{k}\right)^{-\frac{1-\g}{\g}}+b\dfrac{\n}{k}\right)+\e\n\right]d\nu\\
+&\int_{kb^{-\g}}^{\infty}\nu^{-n_1-1}\left[\dfrac{\g}{1-\g}\nu^{-\frac{1-\g}{\g}}-l-\dfrac{b^{1-\g}}{1-\g}+\e\nu\right]d\nu\\
=&\left[\dfrac{\g}{1-\g}\dfrac{1}{1-n_1-\frac{1}{\g}}\left(1-k^{\frac{1-\g}{\g}}\right)(kb^{-\g})^{1-n_1-\frac{1}{\g}}+l\dfrac{(kb^{-\g})^{-n_1}}{n_1}-\dfrac{(kb^{-\g})^{1-n_1}}{n_1-1}\left(\e-\dfrac{b}{k}\right)\right]\\
-&\left[\dfrac{\g}{1-\g}\dfrac{1}{1-n_1-\frac{1}{\g}}\left(1-k^{\frac{1-\g}{\g}}\right)(z_R)^{1-n_1-\frac{1}{\g}}+l\dfrac{(z_R)^{-n_1}}{n_1}-\dfrac{(z_R)^{1-n_1}}{n_1-1}\left(\e-\dfrac{b}{k}\right)\right]\\
-&\left[\dfrac{\g}{1-\g}\dfrac{1}{1-n_1-\frac{1}{\g}}(kb^{-\g})^{1-n_1-\frac{1}{\g}}+l\dfrac{(kb^{-\g})^{-n_1}}{n_1}+\dfrac{b^{1-\g}}{1-\g}\dfrac{(kb^{-\g})^{-n_1}}{n_1}-\e\dfrac{(kb^{-\g})^{1-n_1}}{n_1-1}\right].
\end{align*}

\item[(Case 2)] ${\cal G}(kb^{-\g})\le 0$ (or equivalently, $z_R\ge kb^{-\g}$)\\

\noindent In this case, the coefficient $D$ is given by 
\begin{align*}
D=&-\dfrac{2}{\t^2(n_1-n_2)}\int_0^{z_R}\nu^{-n_2}\Psi(\nu)d\nu\\
=&-\dfrac{2}{\t^2(n_1-n_2)}\int_0^{kb^{-\g}}\nu^{-n_2-1}\left[\dfrac{\g}{1-\g}\n^{-\frac{1-\g}{\g}}-l-\left(\dfrac{\g}{1-\g}\left(\dfrac{\n}{k}\right)^{-\frac{1-\g}{\g}}+b\dfrac{\n}{k}\right)+\e\n\right]d\nu\\
&-\dfrac{2}{\t^2(n_1-n_2)}\int_{kb^{-\g}}^{z_R}\nu^{-n_2-1}\left[\dfrac{\g}{1-\g}\nu^{-\frac{1-\g}{\g}}-l-\dfrac{b^{1-\g}}{1-\g}+\e\nu\right]d\nu\\
=&-\dfrac{2}{\t^2(n_1-n_2)}\left[\dfrac{\g}{1-\g}\dfrac{1}{1-n_2-\frac{1}{\g}}\left(1-k^{\frac{1-\g}{\g}}\right)(kb^{-\g})^{1-n_2-\frac{1}{\g}}+l\dfrac{(kb^{-\g})^{-n_2}}{n_2}-\dfrac{(kb^{-\g})^{1-n_2}}{n_2-1}\left(\e-\dfrac{b}{k}\right)\right]\\
&-\dfrac{2}{\t^2(n_1-n_2)}\left[\dfrac{\g}{1-\g}\dfrac{1}{1-n_2-\frac{1}{\g}}(z_R)^{1-n_2-\frac{1}{\g}}+l\dfrac{(z_R)^{-n_2}}{n_2}+\dfrac{b^{1-\g}}{1-\g}\dfrac{(z_R)^{-n_2}}{n_2}-\e\dfrac{(z_R)^{1-n_2}}{n_2-1}\right]\\
&+-\dfrac{2}{\t^2(n_1-n_2)}\left[\dfrac{\g}{1-\g}\dfrac{1}{1-n_2-\frac{1}{\g}}(kb^{-\g})^{1-n_2-\frac{1}{\g}}+l\dfrac{(kb^{-\g})^{-n_2}}{n_2}+\dfrac{b^{1-\g}}{1-\g}\dfrac{(kb^{-\g})^{-n_2}}{n_2}-\e\dfrac{(kb^{-\g})^{1-n_2}}{n_2-1}\right]
\end{align*}
and $z_R$ is a unique solution of the following algebraic equation:
\begin{align*}
0={\cal G}(z_R)=&\int_{z_R}^\infty \nu^{-n_1}\Psi(\nu)d\nu\\
=&\int_{z_R}^{\infty}\nu^{-n_1-1}\left[\dfrac{\g}{1-\g}\nu^{-\frac{1-\g}{\g}}-l-\dfrac{b^{1-\g}}{1-\g}+\e\nu\right]d\nu\\
=&-\left[\dfrac{\g}{1-\g}\dfrac{1}{1-n_1-\frac{1}{\g}}(z_R)^{1-n_1-\frac{1}{\g}}+l\dfrac{(z_R)^{-n_1}}{n_1}+\dfrac{b^{1-\g}}{1-\g}\dfrac{(z_R)^{-n_1}}{n_1}-\e\dfrac{(z_R)^{1-n_1}}{n_1-1}\right].
\end{align*}

}

\section{Concluding Remarks}\label{sec:conclusion}

 We have studied a retirement and consumption/portfolio choice of an individual economic agent. We have identified minimum assumptions for the problem well-defined and the retirement option valuable. We have solved the problem by formulating a Lagrangian which allows us to consider the difference between the dual value functions before retirement and after retirement as the option value of lifetime labor. We have shown that the option value is the present value of the marginal benefit of work compared to retirement. We have shown human wealth is a component in the derivative of the option value with respect to the marginal value of wealth.
 
 In this paper we have assumed a constant investment opportunity facing the agent. Consideration of the model to allow general market environment as in \citet{YKS} would be a useful topic for future research. We have not considered borrowing constraints, which can play an important role for individuals' choice. The extension incorporating borrowing constraints would also be an interesting topic for future research.

\bibliographystyle{apalike}
\bibliography{library}

\appendix

\section{Proof of Proposition \ref{pro:budget-equaltiy}. }

 We derive budget constraint \eqref{static_budget_constraint}.
For a given $T>0$, we define the equivalent martingale measure $\mathbb{Q}$ by 
\begin{equation}\label{eq:measure-Q}
\dfrac{d \mathbb{Q}}{d \mathbb{P}}=e^{-\frac{1}{2}\t^2T - \t B_T}.
\end{equation}
Then, $B_t^{\mathbb{Q}}=B_t + \t t, t\in[0,T]$ is a standard Brownian motion under measure $\mathbb{Q}$.
Let us consider the process $N_t:= \int_0^t e^{-rs}(c_s- \e {\bf 1}_{\{s<\tau\}})ds +e^{-rt}X_t$.

Then, It\^o's lemma implies $N_t$ is a continuous $\mathbb{Q}$-local martingale, i.e.,
\begin{equation}
dN_t  =e^{-rt}\s\pi_t dB_t^{\mathbb{Q}}.
\end{equation}

By constraint \eqref{eq:credit_constraint} we know
$$
N_t \ge -\e \int_0^{t} e^{-rs} ds -e^{-rt} \frac{\e}{r}=-\dfrac{\e}{r}.
$$

Since $N_t$ is bounded below, $N_t$ is a $\mathbb{Q}$-supermartingale. It follows that 
\begin{equation}
x=\mathbb{E}^{\mathbb{Q}}[N_0]\ge \mathbb{E}^{Q}[N_T]
\end{equation}

Thus, we have 
\begin{align}
x=X_0 &\geq \mathbb{E}^{\mathbb{Q}} \left[\int_0^T e^{-rt} (c_t- \e {\bf 1}_{\{t<\tau\}})ds\right] +\mathbb{E}^{\mathbb{Q}} \left[e^{-rT} X_T\right]\nonumber \\
&= \mathbb{E}\left[\int_0^T \xi_t (c_t- \e {\bf 1}_{\{t<\tau\}})ds\right] +\mathbb{E}\left[\xi_T X_T\right]\nonumber \\
&\geq \mathbb{E} \left[\int_0^T \xi_t(c_t- \e {\bf 1}_{\{t<\tau\}})ds\right] -\mathbb{E} \left[\xi_T \frac{\e}{r} \right]\nonumber \\
&=\mathbb{E} \left[\int_0^T \xi_t (c_t- \e {\bf 1}_{\{t<\tau\}})ds\right] -e^{-rT}\frac{\e}{r},\label{budget_inequality}
\end{align}
where the second inequality follows from constraint \eqref{eq:credit_constraint}. By the dominated convergence theorem the first expectation in the right-hand side of \eqref{budget_inequality} converges to $\mathbb{E} \left[\int_0^{\infty} \xi_t (c_t- \e {\bf 1}_{\{t<\tau\}})dt\right]$. Now budget constraint \eqref{static_budget_constraint} follows from \eqref{budget_inequality} by taking the limit as $T\to\infty$.
 
{The converse can be obtained by  slightly modifying the proof of Theorem 9.4 in \citet{KS}.}

\section{Review of Proposition 4.1 in \citet{KMZ}}\label{sec:KMZ}


\begin{pro}[Proposition 4.1 in \citet{KMZ}]~\label{pro:review:KMZ}
	Let $f(y)$ be an arbitrary measurable function defined on $(0,\infty)$. Then the following conditions are equivalent:
	\begin{itemize}
		\item[(i)] for every $y>0$ 
		$$
		\mathbb{E}\left[\int_0^\infty e^{-\rho t}|f({\cal Y}_t)|dt \right]<\infty,
		$$
		\item[(ii)] for every $y>0$ 
		$$
		\int_0^y \nu^{-n_2-1} |f(\nu)|d\nu +\int_y^\infty \nu^{-n_1-1}|f(\nu)|d\nu <\infty. 
		$$
	\end{itemize}
	Let us denote $\Xi_f(y)$ by 
	$$
	\Xi_f(y) = \mathbb{E}\left[\int_0^\infty e^{-\rho t} f({\cal Y}_t)dt\right]. 
	$$
	Under the condition (i) or (ii), the following statements are true:
	\begin{itemize}
		\item[(a)] $\liminf_{y\downarrow 0}y^{-n_2}|f(y)|=\liminf_{y\uparrow \infty}y^{-n_1}|f(y)|=0$,
		\item[(b)] $\Xi_f$ has a following form: 
		\begin{footnotesize}
			\begin{align*}
			\Xi_f (y) = \dfrac{2}{\t^2(n_1-n_2)}\left[y^{n_2}\int_0^y \nu^{-n_2-1}f(\nu)d\nu +y^{n_1}\int_y^\infty \nu^{-n_1-1}f(\nu)d\nu\right],
			\end{align*}
		\end{footnotesize}
		\item[(c)] $\Xi_f$ is twice differentiable and 
		$$
		\dfrac{\t^2}{2}y^2\Xi_f''(y) + (\rho-r)y\Xi_f'(y) -\rho \Xi_f(y) +f(y)=0,
		$$
		\item[(d)] there exists a positive constant $C$ such that 
		$$
		|\Xi_f'(y)|\le C(y^{n_1-1}+y^{n_2-1})\;\;\;\mbox{for all}\;\;y>0,
		$$
		\item[(e)] $\lim_{t\to \infty} e^{-\rho t} \mathbb{E}\left[|\Xi_f({\cal Y}_t)|\right]=0$. 
	\end{itemize}
\end{pro}

\end{document}